\documentclass[10pt]{article}
\usepackage{amsmath, amssymb ,amsthm, amsfonts, amsgen}
\usepackage{graphicx}
\usepackage[dvips]{color}

\numberwithin{equation}{section} 

\newcommand{\Rb}{{\mathbb{R}}}

\newcommand{\C}{{\mathcal{C}}}

\newcommand{\LL}{{\mathcal{L}}}
\newcommand{\HH}{{\mathcal{H}}}
\newcommand{\M}{{\mathcal{M}}}

\makeatletter
\def\rightharpoonupfill@{\arrowfill@\relbar\relbar\rightharpoonup}
\newcommand{\xrightharpoonup}[2][]{\ext@arrow
0359\rightharpoonupfill@{#1}{#2}} \makeatother

\newcommand{\res}{\mathop{\hbox{\vrule height 7pt width .5pt depth 0pt
\vrule height .5pt width 6pt depth 0pt}}\nolimits}

\def\weakstar{\buildrel\star\over\rightharpoonup}
\def\e{{\varepsilon}}
\def\O{{\Omega}}

\def\weak{\rightharpoonup}

\def\M{{\it M}}

\def\M{{\cal M}}

\newtheorem{Theorem}{Theorem}[section]
\newtheorem{Lemma}[Theorem]{Lemma}
\newtheorem{Proposition}[Theorem]{Proposition}
\newtheorem{Corollary}[Theorem]{Corollary}

\newtheorem{Remark}[Theorem]{Remark}
\newtheorem{Definition}[Theorem]{Definition}

\newcommand{{\rr}}{{\mathbb R}}

\newenvironment{@abssec}[1]{%
     \if@twocolumn
       \section*{#1}%
     \else
       \vspace{.05in}\footnotesize
       \parindent .2in
         {\upshape\bfseries #1. }\ignorespaces
     \fi}
     {\if@twocolumn\else\par\vspace{.1in}\fi}
\newcommand\keywordsname{Key words}
\newcommand\AMSname{AMS subject classifications}

\begin{document}

\title{Relaxation of certain integral functionals depending on strain and chemical composition}
\author{\textsc{Ana Margarida Ribeiro}\thanks{Departamento de Matem\'{a}tica and CMA, Faculdade de Ci\^encias e Tecnologia - Universidade
Nova de Lisboa, Quinta da Torre, 2829-516 Caparica, Portugal. E-mail: amfr@fct.unl.pt},\textsc{ Elvira Zappale}\thanks{D.I.E.I.I, Universita'
degli Studi di Salerno, Via Ponte Don Melillo, 84084 Fisciano (SA) Italy.
E-mail:ezappale@unisa.it}}
\maketitle

\begin{abstract}
We provide a relaxation result in $BV \times L^q$, $1\leq q < +\infty$ as a first step towards the analysis of thermochemical equilibria.



\end{abstract}

\section{Introduction}

In this paper we consider energies depending on two vector fields with different behaviours: $u \in W^{1,1}(\Omega; \mathbb R^d)\cap L^p(\Omega;\mathbb R^d)$, $v \in L^q(\Omega;\mathbb R^m)$, $\Omega$ being a bounded open set of $\mathbb R^N$.
The functional $I:BV(\Omega;\mathbb{R}^d)\times L^q(\Omega;\mathbb{R}^m)\longrightarrow\mathbb{R}\cup\{+\infty\}$ that we consider is defined by
\begin{equation}\label{firstmodel}
I(u,v)=
\left\{ \begin{array}{ll}
\displaystyle{\int_\O W(x, u(x), \nabla u(x))\,dx + \int_\Omega \varphi(x,u(x),v(x))\,dx}\\
\\
\displaystyle{\quad \quad \quad\hbox{ if } (u,v)\in (W^{1,1}(\O;\mathbb{R}^d)\cap L^p(\Omega;\mathbb{R}^d))\times L^q(\O;\mathbb{R}^m),}\vspace{0.2cm}\\+\infty\quad\text{otherwise},\end{array}\right.
\end{equation}
where $W: \Omega \times \mathbb R^d \times \mathbb R^{d \times N} \to \mathbb{R}$ is a continuous function with linear growth from above and below in the gradient variable, $\varphi : \Omega \times \mathbb R^d \times \mathbb R^m\to \mathbb R$ is a  Carath\'eodory function (that is $\varphi(\cdot,u,v)$ is measurable for all $(u,v)\in \mathbb R^d \times \mathbb R^m$ and $\varphi(x,\cdot,\cdot)$ is continuous for a.e. $x\in\O$), with growth $p$ and $q$ respectively in the variables $u$ and $v$.

Our results can be considered as a first step towards the analysis of functionals of the type \hfill

\noindent $\int_\Omega V(x,u,\nabla u, v)dx$, which generalizes those considered by \cite{Fonseca-Kinderlehrer-Pedregal_1},\cite{Fonseca-Kinderlehrer-Pedregal_2}  and \cite{CRZ2}, to deal with equilibria for systems depending on elastic strain and chemical composition. In this context a multiphase alloy is represented by the set $\Omega$, the deformation gradient is given by $\nabla u$, and $v$ denotes the chemical composition of the system.

In \cite{Fonseca-Kinderlehrer-Pedregal_1}, $V \equiv V(\nabla u, v)$ is a cross-quasiconvex function, while in our decoupled model we also take into account heterogeneities and the deformation without imposing any convexity restriction neither on $W$ nor on $\varphi$. Moreover when $\varphi\equiv 0$, the functional in \eqref{firstmodel} recovers the one in
\cite{Fonseca-Muller-relaxation} without quasiconvexity assumptions.

Additive models like the one we are addressing can also be found in imaging models, like those considered in \cite{AAB-FC1, AAB-FC2, AK}, i.e.
\begin{equation}\label{AAB-FC}
\displaystyle{\inf_{u ,v}\left\{|Du|(\O) + \frac{1}{2 \lambda} \|\phi- u - v\|^2_{L^2}\right\}}
\end{equation}
where $\phi$ is a given image and $\lambda$ a scaling factor for the $L^2$ norm of the fidelity term $\phi-(u+v)$.

In order to deal with the minimization of \eqref{firstmodel}, since there may be a lack of lower semicontinuity, it is necessary to pass to the relaxed functional defined in $BV(\Omega;\mathbb{R}^d)\times L^q(\Omega;\mathbb{R}^m)$
\begin{equation}\label{relaxedfunctional1}
\overline{I}(u,v):= \inf \left\{\liminf_{n \to + \infty}I(u_n,v_n): (u_n,v_n) \in BV(\O;\mathbb R^d)\times L^q(\O; \mathbb R^m): u_n \to u \hbox{ in } L^1, v_n \weak v  \hbox{ in } L^q \right\},
\end{equation}
and prove a representation result for $\overline{I}$.

It is worthwhile to remark that for $q=1$, the functional $\overline{I}$ may fail to be sequentially lower semicontinuous. However, as we will observe below, this can be achieved provided that $\varphi$ is uniform continuous, cf. \eqref{uniformcontinuityvarphi}.

We prove the following theorem.
\begin{Theorem}\label{thmdec1}
Let $p\geq 1$ and $q \geq 1$ and let $\Omega \subset \mathbb R^N$ be a bounded open set. Assume that $W:\Omega \times \mathbb R^d \times \mathbb R^{d \times N} \to \mathbb R$ is a continuous function, satisfying
\begin{itemize}
\item [(i)] $\displaystyle\exists \ C>0:\ \frac{1}{C}|\xi|-{C}\leq W(x,u,\xi)\leq C(1+ |\xi|),\ \forall\ (x,u,\xi) \in \Omega \times \mathbb R^d \times \mathbb R^{d \times N};$

\item [(ii)] for every compact subset $K$ of $\Omega \times \mathbb R^d$ there is a  continuous function $\omega_K:[0,+\infty)\rightarrow \mathbb{R}$ with $\omega_K(0)=0$, and such that
$$\displaystyle{\left|W(x,u,\xi)-W(x',u',\xi)\right|\leq \omega_K(|x-x'|+|u-u'|)(1+ |\xi|),\ \forall\ (x,u,\xi), (x',u',\xi) \in K\times \mathbb R^{d \times N};}$$

\item [(iii)] for every $x_0 \in \O$  and for every $\e>0$ there exists $\delta>0$ such that
$$|x-x_0| < \delta\ \Rightarrow\  W(x,u,\xi)-W(x_0,u,\xi) \geq -\e(1+ |\xi|),\ \forall\ (u,\xi)\in \mathbb R^d \times \mathbb R^{d \times N};$$

\item [(iv)] there exist $\alpha \in (0,1),$ and $C,L>0$ such that
$$ t|\xi|>L\ \Rightarrow\ \left|W^\infty(x,u,\xi)- \frac{W(x,u, t \xi)}{t}\right|\leq C \frac{|\xi|^{1-\alpha}}{t^\alpha},\ \forall\ (x,u,\xi) \in \Omega \times \mathbb R^d \times \mathbb R^{d \times N}, t \in \mathbb R.$$
\end{itemize}

Moreover let $\varphi:\Omega \times \mathbb R^d \times \mathbb R^m\to \mathbb R$ be a Carath\'eodory function, satisfying
\begin{itemize}
\item [(v)] $\displaystyle\exists\ C>0: \; \frac{1}{C}(|u|^p+|v|^q)-C\le \varphi(x,u,v)\le C(1+|u|^p+|v|^q),$  $\forall\ (x,u,v) \in \Omega\times\mathbb R^d \times \mathbb R^m$.
\end{itemize}

If $I$ is defined by \eqref{firstmodel} and $\overline{I}$ is defined by \eqref{relaxedfunctional1} then, for every $u \in BV(\Omega; \mathbb R^d)\cap L^p(\Omega;\mathbb R^d)$ and $v \in L^q(\Omega;\mathbb R^m)$, the following identity holds:
\begin{equation}\label{representation1}
\begin{array}{rcl}\overline{I}(u,v) & = &\displaystyle\int_\Omega QW(x,u(x),\nabla u(x))\,dx + \int_{J_u}\gamma(x,u^+, u^-, \nu_u)\,d {\cal H}^{N-1}+  \vspace{0.2cm}\\& & \displaystyle +\int_\O (QW)^\infty \left(x,u(x), \frac{d D^c u}{d |D^c u|}\right)d |D^c u|+\int_{\Omega} C\varphi(x,u(x),v(x))\,dx.\end{array}
\end{equation}
\end{Theorem}

\begin{Remark}\label{firstremark}\smallskip

\noindent\begin{itemize}
\item[i)] An example of an integrand $W$ satisfying the assumptions of Theorem \ref{thmdec1} is given, by $W(x,u,F):= f(x)h(u)g(F)$, where $f:\Omega \subset \mathbb R^2 \to \mathbb R$ and $h:\mathbb R^2 \to \mathbb R$ are continuous bounded functions, bounded from below by a strictly positive constant, $g:\mathbb R^{2 \times 2}\to \mathbb R$ where $g(F):= |F_{11}-F_{22}|+ |F_{12}+ F_{21}|+ \min\{|F_{11}+F_{22}|, |F_{12}-F_{21}|\}$ is the function in \cite[Example 2.18]{Fonseca-Muller-relaxation}, which is not quasiconvex. For what concerns $\varphi$ we can take $\varphi(x,u,v)\equiv f(x) (|u|^p + g_1(v))$, with $f$ as above and  $g_1:\mathbb R^m \to \mathbb R$  any double well function with the required growth, as for example $g_1(v)=(|v|-1)^p$.

\item[ii)] In order to describe the right hand side of (\ref{representation1}) we recall that for every $x \in \Omega$, $QW(x,u,\cdot)$ stands for the quasiconvexification of $W$, cf. (\ref{qcx-rel}), while $(QW)^\infty$ denotes the recession function of $QW$ with respect to the last variable as introduced in Definition \ref{recessiondef}, and $\gamma$ stands for the surface integral density, defined in \eqref{gammaAFP}.  Finally for every $(x, u) \in \Omega \times \mathbb R^d$,  $C\varphi$ stands for the convex envelope (or convexification) of $\varphi(x,u, \cdot)$, namely
\begin{equation}\label{cx-rel}
\displaystyle{C\varphi(x,u,\cdot):= \sup\{g:\mathbb R^m \to \mathbb R: g \hbox{ convex, } g(v)\leq \varphi(x,u,v)\; \forall \,v\}.}
\end{equation}

Classical results in Calculus of Variations ensure that, if $\varphi$ takes only finite values then $C\varphi$ coincides with the bidual of $\varphi$, $\varphi^{\ast \ast}$,  whose characterization is given below
\begin{equation}\label{convexification1}
\varphi^{\ast \ast}(x,u,\cdot):= \sup \{ g: \mathbb R^m \to \mathbb R: g \hbox{ convex and lower semicontinuous},\  g(v)\leq \varphi(x,u,v) \; \forall \, v\}.
\end{equation}

\item[iii)] We observe that if $\varphi \equiv 0$ our results extends \cite[Theorem 2.16]{Fonseca-Muller-relaxation} (see also \cite[Theorem 5.54]{AFP}) to non quasiconvex functions. We stress the fact that our hypotheses are made on the non quasiconvex function $W$ and thus we can't immediately apply the results in \cite{Fonseca-Muller-relaxation} to $Q W$.
\end{itemize}
\end{Remark}

\begin{Remark}\label{mainobservations}

\smallskip
\noindent\begin{itemize}
\item We observe that in the Sobolev setting, Theorem \ref{thmdec1} can be proven without coercivity assumptions on $\varphi$, indeed let $f:\Omega\times\mathbb{R}^{d}\times\mathbb{R}^{d\times N}\times
\mathbb{R}^{m}\rightarrow\mathbb{R}$ be a Carath\'{e}odory function satisfying
$$
0\leq f\left(  x,u,\xi,v\right)  \leq C\left(  1+\left\vert u\right\vert
^{p}+\left\vert \xi\right\vert ^{p}+\left\vert v\right\vert ^{q}\right)
$$
for a.e. $x\in\Omega,$ for every $\left(  u, \xi,v\right)  \in\mathbb{R}%
^{d}\times\mathbb{R}^{d\times N}\times\mathbb{R}^{m}$ and for some $C>0.$
Consider for every $1\leq p, q<+\infty$   the following relaxed localized energy

\begin{equation}\label{calFCRZ2}
\begin{array}{ll}
\displaystyle{\mathcal{F}\left(  u,v;A\right) :=}\\
\\
 \quad \displaystyle{\inf\left\{  \underset{n\rightarrow\infty
}{\lim\inf}\int_{A}f\left(  x,u_{n}\left(  x\right)  ,\nabla u_{n}\left(
x\right)  ,v_{n}\left(  x\right)  \right)  dx:u_{n}\rightharpoonup u\text{ in
}W^{1,p}\left(  A;\mathbb{R}^{d}\right)  ,~v_{n}\rightharpoonup v\text{ in
}L^{q}\left(  A;\mathbb{R}^{m}\right)  \right\}  .}
\end{array}
\end{equation}
Then, in \cite[Theorem 1.1]{CRZ2}(cf. also \cite{CRZ}) it has been proven that, for every $u\in W^{1,p}\left(  \Omega;\mathbb{R}^{d}\right)  ,$ $v\in
L^{q}\left(  \Omega;\mathbb{R}^{m}\right)  $ and $A\in\mathcal{A}\left(
\Omega\right),$
\[
\mathcal{F}\left(  u,v;A\right)  =\int_{A}QCf\left(  x,u\left(  x\right)
,\nabla u\left(  x\right)  ,v\left(  x\right)  \right)  dx,
\]
where $QCf$ stands for the quasiconvex-convex envelope of $f$ with respect to
the last two variables, namely
\begin{equation}
\begin{array}{ll}
\displaystyle{QCf(x,u,\xi,v)=}\\
\\
\displaystyle{\inf\left\{  \frac{1}{|D|}\int_{D}f(x,u,\xi+\nabla\varphi
(y),v+\eta(y))\,dy:\ \varphi\in W_{0}^{1,\infty}(D;\mathbb{R}^{d}),\eta\in
L^{\infty}(D;\mathbb{R}^{m}),\int_{D}\eta(y)\,dy=0\right\},}
\label{qcx-cx-rel}%
\end{array}
\end{equation}
$D$ being any bounded open set. Clearly this equality recovers our setting, since it suffices to define for every $(x,u,\xi,v)\in \Omega \times \mathbb R^d \times \mathbb R^{d \times N}\times \mathbb R^m$,
$ f(x,u,\xi,v):= W(x,u,\xi) + \varphi(x,u,v)$. In fact it is easily seen that if $f$ satisfies the above growth assumptions, then
$$
QCf(x,u,\xi,v)= QW(x,u,\xi)+ C\varphi(x,u,v).
$$

\item We notice that contrary to what one would expect from \cite{Fonseca-Muller-relaxation} and \cite{Fonseca-Kinderlehrer-Pedregal_2}, our density $\varphi$ does not need to satisfy a property analogous to $ii)$ in Theorem \ref{thmdec1} with respect to $(x,u,v)$, indeed it is just a Carath\'eodory function.

\item We emphasize that the arguments adopted to prove the previous theorem strongly rely on the fact that the energy densities are decoupled.
In particular, in the case $q=1$, we will approximate the functional $I$ by adding an extra term with superlinear growth at $\infty$ in the $v$ variable. This will ensure the sequentially weak lower semicontinuity of the relaxed approximating functional $$\begin{array}{l}\displaystyle\overline{I_\varepsilon}(u,v):=\inf\left\{\liminf_n \int_\Omega W(x,u_n,\nabla u_n)\,dx+ \int_\Omega(\varphi(x,u,v)+ \varepsilon \theta(|v|))\,dx:\right.\vspace{0.2cm}\\ \displaystyle\hspace{2.5cm}\left. u_n \to u \hbox{ in }L^1, v_n \rightharpoonup v \hbox{ in }L^1\right\},\end{array}$$ allowing us to adopt arguments similar to those exploited in the proof for the case $q>1$.
These techniques are well suited for the convex setting but we are not aware if a similar procedure is possible in the quasiconvex-convex framework.
\end{itemize}
\end{Remark}

Having in mind the continuous embedding of $BV(\Omega;\mathbb R^d)$ in $L^\frac{N}{N-1}(\Omega;\mathbb R^d)$ (assuming $\Omega \subset \mathbb R^N$), we can obtain, in an easier way, the relaxation result as above.
%for models whose energy we describe next.
%Let $1\leq p \leq \frac{N}{N-1}$ and $q\geq1$, assume that $W:\Omega \times \mathbb R^n %\times \mathbb R^{n \times N} \to \mathbb R$ satisfies the same assumptions as in Theorem %\ref{thmdec1} and $\varphi:\Omega \times \mathbb R^n \times \mathbb R^m \to \mathbb R$ is %a Carath\'eodory function satisfying growth condition of order $p$ from above and below in the %variable $v$ and of order $q$ on $u$ from above.
Indeed we can prove the following result.
%\begin{equation}\label{Jbardec}
%\displaystyle{\overline{J}(u,v):=\inf\left\{\liminf_{n \to +\infty} J(u_n,v_n): (u_n, v_n)\in %BV(\Omega;\mathbb R^d)\times L^q(\Omega;\mathbb R^m): u_n \to u \hbox{ in }L^1, v_n %\rightharpoonup v \hbox{ in }L^q\right\},}
%\end{equation}
%proving the following theorem.

\begin{Theorem}\label{thmdec2}
Let $\Omega \subset \mathbb R^N$ be a bounded open set, and let $1\leq p\le \frac{N}{N-1}$ and $q\geq1$. Let $W:\Omega \times \mathbb R^d \times \mathbb R^{d \times N} \to \mathbb R$ be a continuous function satisfying $(i) \div (iv)$ of Theorem \ref{thmdec1}.
Moreover let $\varphi:\Omega \times \mathbb R^d \times \mathbb R^m \to \mathbb R$ be a Carath\'eodory function satisfying $(v)$ of Theorem \ref{thmdec1} in the weaker form
\begin{equation}\label{varphigrowth2}\exists\ C>0:\ \displaystyle{ \frac{1}{C}|v|^q-C \leq \varphi(x,u,v)\leq C(1+ |u|^p+ |v|^q),\ \forall\ (x,u,v)\in \Omega \times \mathbb R^d \times \mathbb R^m.}
\end{equation}

Then, for every $(u,v) \in BV(\Omega;\mathbb R^d) \times L^q(\Omega;\mathbb R^m),$ \eqref{representation1} holds.

%\begin{equation}\label{Jdecrep2}
%\begin{array}{rcl}\overline{J}(u,v) & = &\displaystyle\int_\Omega QW(x,u(x),\nabla u(x))\,dx + %\int_{J_u}\gamma(x,u^+(x), u^-(x), \nu_u(x))\,d {\cal H}^{N-1}(x)+  \vspace{0.2cm}\\& & %\displaystyle +\int_\O (QW)^\infty \left(x,u(x), \frac{d D^c u}{d |D^c u|}(x)\right)d |D^c u|%(x)+\int_{\Omega} C\varphi(x,u(x),v(x))\,dx.\end{array}
%\end{equation}
\end{Theorem}

\begin{Remark}

%(i) We emphasize that the decoupled model considered in this Theorem is interesting in itself and %it is not a mere corollary of the general  model presented in the case $BV \times L^p$: the %coercivity condition assumed on $\varphi$ is less restrictive than the one of Theorem %\ref{thmdec1}, but the values of $p$ are more restrictive. Moreover, since in this case we have %no need of using Reshetnyak theorem, we can assume a dependence on $u$ in the function %$W$.

Theorem \ref{thmdec2} allows us to consider very general growth conditions also in \cite[Theorem 1.1]{CRZ2} when decoupled energies are considered.

 The continuous embedding of $BV(\O;\mathbb R^d)$ into $L^{\frac{N}{N-1}}(\Omega;\mathbb R^d)$ (with $\Omega \subset \mathbb R^N$) allows us to obtain Theorem \ref{thmdec2}, also replacing \eqref{varphigrowth2} by the following condition:
$$
\exists\ C>0:\ \frac{1}{C}|v|^q-C\leq \varphi(x,u,v)\leq C(1+ |u|^r+ |v|^q),\ \forall \ (x,u,v)\in \Omega\times\mathbb R^d \times \mathbb R^m
$$
and for some $r \in \left[1,\frac{N}{N-1}\right]$.

%(iii) In order to describe the right hand side of (\ref{Jdecrep})  we refer, as for Theorem  %\ref{thmdec1} to formulae (\ref{convexification1}), (\ref{qcx-rel}), to Definition %\ref{recessiondef}, and finally we observe that the surface energy density $\gamma$ is defined %in (\ref{gammaAFP}) below.
\end{Remark}

We observe that under assumptions $(i)\div (iv)$ of Theorem \ref{thmdec1}, \cite[Theorem 2.16]{Fonseca-Muller-relaxation} ensures that the functional
$$
\int_\O QW(x,u,\nabla u)\,dx + \int_{J_u} \gamma(x,u^+,u^-,\nu_u)\,d {\cal H}^{N-1}+
\int_\O (QW)^\infty\left(x,u, \frac{d D^c u}{d|D^cu|}\right)d |D^c u|
$$
is lower semicontinuous with respect to the strong-$L^1$ topology.
Moreover \cite[Theorem 7.5]{Fonseca-Leoni} guarantees that
$$
\int_\Omega C\varphi(x,u,v)\,dx
$$
is sequentially weakly lower semicontinuous with respect to $L^1_{\rm strong}\times L^1_{\rm weak}$-topology provided the function $C\varphi$ is convex in the last variable, satisfies suitable growth conditions, as those in \eqref{cvarphigrowth1} and \eqref{cvarphigrowth2}, and that the function $C\varphi(x,\cdot,\cdot)$ is lower semicontinuous.
We will observe in Remark \ref{measCvarphi} below that this latter condition may not be verified just under the assumptions of Theorems \ref{thmdec1} and \ref{thmdec2}.
On the other hand an argument entirely similar to \cite[Theorem 9.5]{Dacorogna} guarantees that $C\varphi(x,\cdot,\cdot)$ is lower semicontinuous (even continuous) by assuming additionally that
\begin{equation}\label{uniformcontinuityvarphi}
|\varphi(x,u,\xi)- \varphi(x, u',\xi)|\leq \omega'(|u-u'|)(|\xi| +1)
\end{equation}
for a suitable modulus of continuity $\omega'$, i.e. $\omega' :\mathbb R^+\cup\{0\}\to \mathbb R^+\cup\{0\}$ continuous and such that $\omega' (0)=0$.

Consequently the superadditivity of $\liminf$ entails the sequentially strong-weak lower semicontinuity of the right hand side of \eqref{representation1} even for $q=1$.

\section{Notations and General Facts}\label{Notations}

\subsection{Properties of the integral density functions}

In this subsection we recall several notions applied to functions like quasiconvexity,  envelopes and recession function, etc. We also recall or prove properties of those functions that will be useful through the paper. Such notions and related properties will apply to the density functions that will appear in the relaxed functionals that we characterize. We start recalling the notion of quasiconvex function due to Morrey.

\begin{Definition}
A Borel measurable function $h: \mathbb R^{d \times N} \to \mathbb R$ is said to be quasiconvex if there exists a bounded open set $D$ of $\mathbb R^N$ such that
$$h(\xi)\le\frac{1}{|D|}\int_D h(\xi + \nabla \varphi(x))\,dx,$$
for every $\xi\in \mathbb R^{d \times N}$, and for every $\varphi \in W^{1, \infty}_0(D;\mathbb R^{d})$.
\end{Definition}

If $h: \mathbb R^{d \times N} \to \mathbb R$ is any given Borel measurable function bounded from below, it can be defined the quasiconvex envelope of $h$, that is the largest quasiconvex function below $h$:
$$Qh(\xi):=\sup\{g(\xi):\ g \leq h, \ g \hbox{ quasiconvex}\}.
$$
Moreover, as well known (see the monograph \cite{Dacorogna}),
\begin{equation}\label{qcx-rel}
Qh(\xi):=\inf\left\{\frac{1}{|D|}\int_D h(\xi+\nabla \varphi(x))\,dx:\ \varphi\in W_0^{1,\infty}(D;\mathbb{R}^d)\right\},
\end{equation}
for any bounded open set $D\subset\mathbb{R}^{N}$.

\begin{Proposition}\label{measQW}
Let $\Omega \subset \mathbb R^N$ be a bounded open set and
$$
W: \Omega \times \mathbb R^d \times \mathbb R^{d \times N} \to \mathbb [0, +\infty)
$$
be a continuous function. Let $QW$ be the quasiconvexification of $W$  (see (\ref{qcx-rel})). Then the validity of (i) in Theorem \ref{thmdec1} guarantees that there exists  a constant $C>0$ such that
\begin{equation}\label{H1QW}
\frac{1}{C}|\xi|- C \leq QW(x,u,\xi) \leq C(1+ |\xi|),\ \forall\ (x,u,\xi) \in \Omega \times \mathbb R^d  \times \mathbb R^{d \times N}.
\end{equation}

The validity of (i) and (ii) of Theorem \ref{thmdec1} ensures that
for every compact set $K \subset \Omega \times \mathbb R^d$, there exists a continuous function $\omega'_K:\mathbb{R}\to [0, +\infty)$ such that
$\omega'_K(0)=0$ and
\begin{equation}\label{H4QW1}
|QW(x,u,\xi)- QW(x',u',\xi)| \leq \omega'_K(|x-x'|+|u-u'|)(1 +|\xi|),\  \forall\ (x,u),(x',u')\in K,\ \forall\ \xi \in \mathbb R^{d \times N}.
\end{equation}

Conditions (i) and (iii) of Theorem \ref{thmdec1} entail that, for every $x_0 \in \Omega$ and $\e >0$, there exists $\delta>0$ such that
\begin{equation}\label{H4QW2}
|x-x_0|\leq \delta\ \Rightarrow\ \forall\ (u,\xi)\in \mathbb R^d \times \mathbb R^{d \times N} \; \; QW(x,u,\xi) - QW(x_0,u,\xi) \geq -\e(1+|\xi|).
\end{equation}
Moreover, if $W$ satisfies conditions (i) and (ii) of Theorem \ref{thmdec1}, $QW $ is a continuous function.
\end{Proposition}

\begin{Remark}\label{measCvarphi}
%In the case $W=W(x,\xi)$, as it is the case in Theorem \ref{thmdec1}, we obtain an analogous %result under the assumptions (i), (ii), (iii) and (iv) of the mentioned theorem.
Analogous arguments entail that, under hypothesis $(v)$ of Theorems \ref{thmdec1} and \ref{thmdec2}, respectively,
\begin{equation}\label{cvarphigrowth1}
 \exists \, C>0: \, \frac{1}{C}(|u|^p + |v|^q)-C\leq C\varphi(x,u,v)\leq C(1+ |u|^p+ |v|^q), \;\; \forall \, (x, u, v) \in \Omega \times \mathbb R^d \times \mathbb R^m,
\end{equation}
and
\begin{equation}\label{cvarphigrowth2}
\exists \, C>0: \,\frac{1}{C}|v|^q-C\leq C\varphi(x,u,v)\leq C(1+ |u|^p+ |v|^q), \;\; \forall \,(x, u, v) \in \Omega \times \mathbb R^d \times \mathbb R^m.
\end{equation}
On the other hand we emphasize that being $\varphi$ as in Theorems \ref{thmdec1} and \ref{thmdec2}, namely a Carath\'eodory function, this is not enough to guarantee that $C\varphi$ is still a Carath\'eodory function, cf. Example 9.6 in \cite{Dacorogna} and Example 7.14 in \cite{Fonseca-Leoni}.
In particular $C\varphi$ turns out to be measurable in $x$, upper semicontinuous in $u$, convex and hence continuous in $\xi$. Furthermore if $q>1$, \cite[Lemma 4.3]{DM-F-L-M} guarantees that $C\varphi(x, \cdot, \cdot)$ is lower semicontinuous.
\end{Remark}

\begin{proof}[Proof] By definition of the quasiconvex envelope of $W$, it is easily seen that $(i)$ of Theorem \ref{thmdec1} entails (\ref{H1QW}) with the same constant appearing in $(i)$.

Next we prove (\ref{H4QW1}).
Let $K$ be a compact set in $\Omega \times \mathbb R^d$ and take $(x,u), (x',u')\in K$.
Let $\e>0$, then using condition (\ref{qcx-rel}), we find  $\varphi_{\e} \in W^{1,\infty}_0(Q;\mathbb{R}^d)$, $Q$ being the unitary cube, such that
$$
QW(x,u,\xi) \geq - \e+ \int_Q W(x,u, \xi + \nabla \varphi_{\e}(y))\,dy.
$$
Now, we observe that, by virtue of the coercivity condition expressed by $(i)$ of Theorem \ref{thmdec1} and by \eqref{H1QW}, it follows that
$$
\|\xi+ \nabla \varphi_\e\|_{L^1} \leq c(1+|\xi|).
$$
By condition $(ii)$ of Theorem \ref{thmdec1}, for every $(x,u), (x',u') \in K$ and for every $\xi\in  \mathbb R^{d \times N}$, it results
$$
|W(x,u,\xi)-W(x',u',\xi)|\leq \omega_K(|x-x'|+|u-u'|)(1+ |\xi|).
$$
Then we can write the following chain of inequalities:

$$
\begin{array}{ll}
QW (x,u,\xi)  &\displaystyle \geq -  \e + \int_Q W(x,u,\xi+ \nabla \varphi_{\e}(y))\,dy \\
\\
& \geq \displaystyle -\e- \int_Q \lambda(y)\,dy + \int_{Q} W(x',u', \xi+ \nabla \varphi_{\e}(y))\,dy,
\end{array}
$$
where $\lambda(y):= |W(x,u, \xi +\nabla \varphi_{\e}(y))- W(x',u', \xi +\nabla \varphi_{\e}(y))|.$
We therefore get, from the definition of $QW(x',u',\xi)$, that,
%for $|x-x'|+|u-u'|\leq \delta_1$,
$$
\begin{array}{rcl}
\displaystyle{QW(x',u',\xi)-QW(x,u,\xi)} & \leq & \displaystyle{\e + \omega_K(|x-x'|+|u-u'|)(1 + \|\xi +\nabla \varphi_\e\|_{L^1})}\vspace{0.2cm}\\ & \le &
\displaystyle{\e + \omega_K(|x-x'|+|u-u'|)(1+ c(1+|\xi|)).}
\end{array}
$$
Since $\e$ is arbitrarily chosen, and since we can obtain in a similar way the same inequality with $x$ in the place of $x'$, and $u$ in the place of $u'$, we get (\ref{H4QW1}).

In order to prove condition (\ref{H4QW2}), we fix $x_0 \in \Omega$ and $\e>0$. As before, for every $x \in \Omega$ and for every $\sigma>0$, by \eqref{qcx-rel}, the coercivity condition expressed by $(i)$ of Theorem \ref{thmdec1}, and by \eqref{H1QW}, there exist a constant $c>0$ and a function $\varphi_\sigma \in W^{1,\infty}_0(Q;\mathbb R^d)$ such that
$$
QW(x,u,\xi)\geq - \sigma+\int_Q W(x,u,\xi+ \nabla \varphi_\sigma(y))\,dy ,
$$
with $\|\xi +\nabla \varphi_\sigma\|_{L^1} \leq c(1+|\xi|)$.

Thus arguing as above, and exploiting condition $(iii)$ of Theorem \ref{thmdec1}, we have the following chain of inequalities, for $|x-x_0|<\delta$ with $\delta$ as in condition $(iii)$ of Theorem \ref{thmdec1},
$$
\begin{array}{rcl}
QW(x_0,u,\xi) & \leq & \displaystyle\int_Q W(x_0,u,\xi+\nabla \varphi_\sigma(y))\,dy\\
\\& \le & \displaystyle
\int_Q W(x,u, \xi+\nabla\varphi_\sigma(y))\,dy + \e\int_Q (1 +|\xi+\nabla\varphi_\sigma(y)|)\,dy\\
\\ & \le &
QW(x,u,\xi) + \sigma + \e(1 + c(1+|\xi|)).
\end{array}
$$
Thus it suffices to let $\sigma$ go to $0$ in order to achieve the statement.

Finally we prove the continuity of $QW$. We need to show that, for every  $\e >0$ and $(x_0,u_0,\xi_0)\in \Omega \times \mathbb R^d \times \mathbb R^{d \times N}$, there is  $\delta \equiv \delta (\e, x_0, u_0, \xi_0)>0$ such that
\begin{equation}\label{9.12Dacorogna}
|x-x_0|+|u-u_0|+  |\xi-\xi_0| \leq \delta \Rightarrow |QW(x,u,\xi)-QW(x_0,u_0,\xi_0)|\leq \e.
\end{equation}

Let $\e>0$ be fixed. Since $QW$ is quasiconvex on $\xi$, $QW(x_0,u_0,\cdot)$ is continuous and thus we can find $\delta_1=\delta_1(\e,x_0,u_0,\xi_0)>0$ such that
$$
|\xi-\xi_0| \leq \delta_1\ \Rightarrow\ |QW(x_0,u_0,\xi)-QW(x_0,u_0,\xi_0)|\leq \frac{\e}{2}.
$$
Moreover, by virtue of (\ref{H4QW1}), defining $K:=\overline{B}_\sigma(x_0,u_0)$ for some $\sigma>0$ such that $K\subset\Omega\times\mathbb{R}^d$, one has
$$
|\xi-\xi_0| \leq \delta_1\ \Rightarrow\ |QW(x,u,\xi)-QW(x_0,u_0,\xi)|\leq \omega'_{K}(|x-x_0|+|u-u_0|)(1+|\xi_0|+\delta_1).
$$

Since $\omega'_K$ is continuous and $\omega'_K(0)=0$, there is $\delta_2=\delta_2(\e,K,\xi_0)>0$ such that
$$
|x-x_0|+|u-u_0| \leq \delta_2\ \Rightarrow\ \omega'_{K}(|x-x_0|+|u-u_0|)\le \frac{\e}{2(1+|\xi_0|+1)}.
$$

Consequently, by choosing $\delta$ as $\min \{\delta_1, \delta_2\}$,  the above inequalities, and the triangular inequality give indeed (\ref{9.12Dacorogna}).
\end{proof}

We also recall the definition of the recession function.

\begin{Definition}\label{recessiondef} Let $h:\mathbb{R}^{d \times N}\to [0,+\infty)$. The recession function of $h$ is denoted by $h^\infty: \mathbb{R}^{d \times N}\to [0, +\infty)$, and defined as
$$\displaystyle{h^\infty(\xi):=\limsup_{t \to + \infty} \frac{h(t \xi)}{t}}.$$
\end{Definition}

\begin{Remark}\label{remarkonrecession}{\rm (i) Recall that the recession function is a positively one homogeneous function, that is $g(t\xi)=tg(\xi)$ for every $t\ge 0$ and $\xi\in\mathbb{R}^{d\times N}$.

\noindent(ii) Through this paper we will work with functions $W:\Omega \times \mathbb R^d \times \mathbb R^{d \times N}\to [0,+\infty)$ and $W^\infty$ is the recession function with respect to the last variable:
$$\displaystyle{W^\infty(x,u,\xi):=\limsup_{t \to + \infty} \frac{W(x,u,t \xi)}{t}}.$$

We trivially observe that, if $W$ satisfies the growth condition $(i)$ in Theorem \ref{thmdec1}, then $W^\infty$ satisfies $\frac{1}{C}|\xi|\leq W^\infty(x,u,\xi)\le C |\xi|.$

\noindent(iii) As showed in \cite[Remark 2.2 (ii)]{Fonseca-Muller-relaxation}, if a function $h:\mathbb{R}^{d\times N}\longrightarrow [0,+ \infty)$ is quasiconvex and satisfies the growth condition $h(\xi)\le c(1+|\xi|)$, for some $c>0$, then, its recession function is also quasiconvex.}
\end{Remark}

We now describe the surface energy density $\gamma$ appearing in the characterization of $\overline{I}$. Let $W:\Omega \times \mathbb R^d \times \mathbb R^{d \times N}\to \mathbb{R}$. By the notation above $(QW)^\infty$ is the recession function of the quasiconvex envelope of $W$. Then $\gamma:\Omega\times\mathbb{R}^n\times\mathbb{R}^n\times S^{N-1}\longrightarrow\mathbb{R}$ is defined by
\begin{equation}\label{gammaAFP}
\displaystyle{\gamma(x,a,b,\nu)= \inf \left\{\int_{Q_\nu} (QW)^\infty(x,\phi(y),\nabla \phi(y))\,dy : \phi \in {\cal A}(a,b,\nu)\right\},}
\end{equation} where $Q_\nu$ is the unit cube centered at the origin with faces parallel to $\nu,\nu_1, \dots, \nu_{N-1}$, for some orthonormal basis of $\mathbb R^N$, $\{\nu_1,\dots, \nu_{N-1}, \nu\}$, and where
$$
\begin{array}{ll}
\displaystyle{{\cal A}(a,b,\nu):= \left\{ \phi \in W^{1,1}(Q_\nu,\mathbb{R}^d): \phi(y)=a \hbox{ if } <y, \nu>= \frac{1}{2},\ \phi(y)=b \hbox{ if } <y,\nu>= - \frac{1}{2},  \right.}\\
\displaystyle{\left. \hspace{6cm}\phi \hbox{ is }1-\hbox{periodic in the } \nu_1, \dots, \nu_{N-1} \hbox{ directions } \right\}.}
\end{array}
$$

We observe that the function $\gamma$ is the same whether we consider in the set $\mathcal{A}(a,b,\nu)$, $W^{1,1}(Q_\nu,\mathbb{R}^d)$ functions (like in \cite{Fonseca-Rybka} and \cite{Fonseca-Muller-relaxation}) or $W^{1,\infty}(Q_\nu,\mathbb{R}^d)$ functions (like in \cite[page 312]{AFP}). Moreover, if $W$ doesn't depend on $u$, $W:\Omega \times \mathbb R^{d \times N}\to \mathbb{R}$, then $\gamma(x,a,b,\nu)=(QW)^\infty(x,(a-b)\otimes\nu)$ (see \cite[page 313]{AFP}).\bigskip

 Properties of the function $(QW)^\infty$ will be important to get the integral representation of the relaxed functionals under consideration. In particular, a proof entirely similar to \cite[Proposition 3.4]{BZZ} ensures that for every $(x,u,\xi) \in \Omega \times \mathbb R^d \times \mathbb R^{d \times N}, $ $Q(W^\infty)(x,u,\xi)= (QW)^\infty(x,u,\xi)$.

\begin{Proposition}\label{QWinfty=}
Let $W:\Omega \times \mathbb R^d \times \mathbb R^{d \times N}\to [0, + \infty)$ be a continuous function satisfying $(i)$ and $(iv)$ of Theorem \ref{thmdec1}.
Then
\begin{equation}\label{QWinf=}
Q(W^\infty)(x,u,\xi)= (QW)^\infty(x,u,\xi) \;\;\;\hbox{ for every }(x,u,\xi) \in \Omega \times \mathbb R^d \times \mathbb R^{d \times N}.
\end{equation}
\end{Proposition}
\begin{proof}[Proof]
The proof will be achieved by double inequality.

By definition of the quasiconvex envelope and the recession function, one gets $(QW)^\infty\le W^\infty$ and thus $Q(QW)^\infty\le Q(W^\infty).$ Since the recession function of a quasiconvex one is still quasiconvex, under hypothesis (i) of Theorem \ref{thmdec1} (cf. Remark \ref{remarkonrecession} (iii)) it follows that $(QW)^\infty\le Q(W^\infty).$

In order to prove the opposite inequality we start noticing that, since by (i), the function $W$ is bounded from below, we can assume without loss of generality that $W\ge 0$. Then fix $(x,u,\xi)\in \Omega \times \mathbb R^d \times \mathbb R^{d \times N}$ and, for every $t >1$, take $\varphi_t \in W^{1,\infty}_0(Q;\mathbb R^d)$ such that
\begin{equation}\label{3.7BZZ}
\int_Q W(x,u, t \xi + \nabla \varphi_t(y))\,dy \leq QW(x,u, t \xi )+ 1.
\end{equation}
By $(i)$ and \eqref{H1QW} we have that $\left\|\nabla(\frac{1}{t}\varphi_t)\right\|_{L^1(Q)}\leq C $ for a constant independent of $t$ but just on $\xi$.

Defining $\psi_t=\frac{1}{t}\varphi_t$, one has $\psi_t \in W^{1,\infty}_0(Q;\mathbb R^d)$ and thus
$$Q(W^\infty)(x,u, \xi)\leq \int_Q W^\infty(x,u, \xi + \nabla \psi_t(y))\,dy.$$

Let $L$ be the constant appearing in condition $(iv)$ of Theorem \ref{thmdec1}, we split the cube $Q$ in the set $\{y\in Q:\ t|\xi+\nabla \psi_t(y)|\le L\}$ and its complement in $Q$. Then we apply condition $(iv)$ and the growth of $W^\infty$ observed in Remark \ref{remarkonrecession} $(ii)$ to get
$$\begin{array}{rcl}Q(W^\infty)(x,u,\xi) & \leq & \displaystyle \int_Q \left(C\frac{|\xi+\nabla \psi_t|^{1-\alpha}}{t^\alpha}+\frac{W(x,u,t\xi+\nabla \varphi_t(y))}{t}+C\frac{L}{t}\right)\,dy.
\end{array}$$

Applying H\"{o}lder inequality and (\ref{3.7BZZ}), we get
$$\begin{array}{rcl}Q(W^\infty)(x,u,\xi) & \leq & \displaystyle \frac{C}{t^\alpha}\left(\int_Q |\xi+\nabla \psi_t|\,dy\right)^{1-\alpha}+\frac{QW(x,u, t \xi )+ 1}{t}+C\frac{L}{t},
\end{array}$$
and the desired inequality follows by definition of $(QW)^\infty$ and using the fact that $\nabla\psi_t$ has bounded $L^1$ norm, letting $t$ go to $+\infty$.
\end{proof}

The property of $(QW)^\infty$ stated next ensures that $QW$ together with $(QW)^\infty$ satisfy the analogous condition to $(iv)$ of Theorem \ref{thmdec1}.
To this end we first observe, as emphasized in \cite{Fonseca-Muller-relaxation}, that $(iv)$ in Theorem \ref{thmdec1} is equivalent to say that there exist $C >0$ and $\alpha \in (0,1) $ such that
\begin{equation}\label{ivequivalent}
\displaystyle{\left|W^\infty(x,u,\xi)-W(x,u,\xi)\right|\leq C (1+ |\xi|^{1-\alpha})}
\end{equation}
for every $(x,u,\xi) \in \Omega \times \mathbb R^d \times \mathbb R^{d \times N}.$
Precisely we have the following result.

\begin{Proposition}\label{propperH5}
Let $W:\Omega \times \mathbb R^d \times \mathbb R^{d \times N}\to [0, + \infty)$ be a continuous function satisfying $(i)$ and $(iv)$  of Theorem \ref{thmdec1}.
Then, there exist $\alpha\in (0,1),$ and $C'>0$ such that
$$ \displaystyle{\left| (QW)^\infty(x,u,\xi)- QW(x,u, \xi)\right|\leq C( 1+|\xi|^{1-\alpha}), \;\;\forall\ (x,u, \xi)\in \Omega \times\mathbb{R}^d\times \mathbb R^{d \times N}}.$$
\end{Proposition}
\begin{proof}[Proof]
The thesis will be achieved by double inequality.
Let $\alpha \in (0,1)$ be as in $(iv)$ of Theorem \ref{thmdec1}, see also \eqref{ivequivalent}. Let $\xi \in \mathbb R^{d \times N}$,
%\begin{equation}\label{estimateuseful}
%t|\xi|> ({\cal L}^N(D)+ L)C^2,
%\end{equation} where
let $Q$ be the unit cube in $\mathbb R^N$ and let $c$ be a positive constant varying from line to line.
%growth conditions $(i)$ of Theorem \ref{thmdec1} and in \eqref{H1QW}.
For every $\varepsilon >0$ by \eqref{qcx-rel}, find $\varphi \in W^{1,\infty}_0(Q;\mathbb R^d)$ such that
$$
QW(x, u, \xi) > \int_Q W(x,u, \xi + \nabla\varphi(y))\,dy- \varepsilon.
$$
By $(i)$ of Theorem \ref{thmdec1} and by \eqref{H1QW} there exists
$c>0$ such that
\begin{equation}\label{fromabove}\| \xi + \nabla \varphi\|_{L^1} \leq c(1+| \xi|).
\end{equation}

Since
by Proposition \ref{QWinfty=} it results
$$
(QW)^\infty(x,u, \xi)\leq \int_Q W^\infty(x, u, \xi +\nabla \varphi(y))\,dy,
$$
we have
$$
\begin{array}{ll}
\displaystyle{(QW)^\infty(x,u,\xi)- QW(x,u,\xi) \leq \int _Q \left(W^\infty\left(x, u, \xi +\nabla \varphi(y)\right)- W(x,u,  \xi + \nabla \varphi(y))\right)\,dy + \varepsilon}.
\end{array}
$$
Applying \eqref{ivequivalent}, we obtain
$$
\begin{array}{lll}
(QW)^\infty(x,u,\xi)- QW(x,u, \xi) &\leq\displaystyle{\int _Q c\left(1+ \left|{\xi +\nabla \varphi}(y)\right|^{1-\alpha}\right)\,dy +\varepsilon}\\
\\
&\displaystyle{\leq c \left(1+ \int_Q|\xi + \nabla \varphi(y)|^{1-\alpha}dy \right)+ \varepsilon }\\
\\
&\displaystyle{\leq c + c\left(\int_Q|\xi + \nabla \varphi(y)|dy\right)^{1-\alpha}+ \varepsilon }\\
\\
&\displaystyle{\leq c + c^2 (1+ |\xi|^{1-\alpha}) +\varepsilon }\\
\\
&\displaystyle{\leq C' (1+ |\xi|^{1-\alpha}) + \varepsilon.}
\end{array}
$$
where in the last lines we have applied Holder inequality, \eqref{fromabove} and we have estimated the term $(1+ |\xi|)^{1-\alpha}$ by separating the cases $|\xi|\leq 1$ and $|\xi |>1$ and summing them up.
To conclude this part it suffices to send $\varepsilon$ to $0$.

In order to prove the opposite inequality we can argue in the same way. Let
 $\xi\in \mathbb R^{d \times N}$. For every $\varepsilon >0$, by \eqref{qcx-rel} and Proposition \ref{QWinfty=} there exists $\psi\in W^{1,\infty}_0(Q;\mathbb R^d)$ such that
$$
\displaystyle{(QW)^\infty(x,u,\xi)>\int_Q W^\infty(x,u, \xi +\nabla \psi(y))\,dy -\varepsilon.}
$$
Clearly, by \eqref{H1QW}, $(i)$ of Theorem \ref{thmdec1} and (ii) of Remark \ref{remarkonrecession} there exists $C>0$ such that
\begin{equation}\label{fromabove2}
\|\xi +\nabla \psi\|_{L^1}\leq C|\xi|+ \varepsilon.
\end{equation}
By \eqref{qcx-rel} it results
$$
QW(x,u, \xi) \leq \int_Q W(x, u, \xi+ \nabla \psi(y))\,dy,
$$
hence
$$
\displaystyle{QW(x,u,\xi)-(QW)^\infty(x,u,\xi)\leq \int_Q \left(W(x,u,  \xi +  \nabla \psi(y)) - W^\infty(x,u,\xi+\nabla \psi(y))\right)\,dy +\varepsilon}.
$$
Now, $(iv)$ of Theorem \ref{thmdec1} in the form \eqref{ivequivalent} provide
$$
\begin{array}{ll}
\displaystyle{QW(x,u,\xi)-(QW)^\infty(x,u,\xi)}&\displaystyle{\leq C\int_Q (1+ |\xi +\nabla \psi(y)|^{1-\alpha})\,dy+ \varepsilon\leq}\\
\\
&\displaystyle{\leq C' (1+ |\xi|^{1-\alpha})+ \varepsilon,}
\end{array}
$$
where in the last line it has been used Holder inequality, \eqref{fromabove2} and an argument entirely similar to the first part of the proof.
%the fact that $\|\xi + \nabla \psi_t\|_{L^1(D)}\leq C|\xi|+ \varepsilon$
By sending $\varepsilon $ to $0$  we conclude the proof.
\end{proof}

\subsection{Some Results on Measure Theory and $BV$ Functions}

Let $\O$ be a generic open subset of $\Rb^N$, we denote by
$\M(\O)$ the space of all signed Radon measures in $\O$ with bounded
total variation. By the Riesz Representation Theorem, $\M(\O)$ can
be identified to the dual of the separable space $\C_0(\O)$ of
continuous functions on $\O$ vanishing on the boundary $\partial
\O$. The $N$-dimensional Lebesgue measure in $\Rb^N$ is designated
as $\LL^N$ while $\HH^{N-1}$ denotes the $(N-1)$-dimensional
Hausdorff measure. If $\mu \in \M(\O)$ and $\lambda \in \M(\O)$ is a
nonnegative Radon measure, we denote by $\frac{d\mu}{d\lambda}$ the
Radon-Nikod\'ym derivative of $\mu$ with respect to $\lambda$. By a
generalization of the Besicovich Differentiation Theorem (see
\cite[Proposition 2.2]{Ambrosio-Dal Maso}), it can be proved that  there exists a
Borel set $E \subset \O$ such that $\lambda(E)=0$ and
\begin{equation}\label{Thm2.7FMr}
\frac{d\mu}{d\lambda}(x)=\lim_{\rho \to 0^+} \frac{\mu(x+\rho \, C)}{\lambda(x+\rho \, C)}\ \text{ for all }x \in {\rm Supp }\,\ \mu \setminus E\end{equation}
and any open convex
set $C$ containing the origin. (Recall that the set $E$ is independent of $C$.)

\vskip5pt

We say that $u \in L^1(\O;\Rb^d)$ is a function of bounded
variation, and we write $u \in BV(\O;\Rb^d)$, if all its first
distributional derivatives $D_j u_i$ belong to $\M(\O)$ for $1\leq i
\leq d$ and $1 \leq j \leq N$. We refer to \cite{AFP} for a detailed
analysis of $BV$ functions. The matrix-valued measure whose entries
are $D_j u_i$ is denoted by $Du$ and $|Du|$ stands for its total
variation. By the Lebesgue Decomposition Theorem we can split $Du$
into the sum of two mutually singular measures $D^a u$ and $D^s u$
where $D^a u$ is the absolutely continuous part of $Du$ with respect
to the Lebesgue measure $\LL^N$, while $D^s u$ is the singular part
of $Du$ with respect to $\LL^N$. By $\nabla u$ we denote the
Radon-Nikod\'ym derivative of $D^au$ with respect to the Lebesgue
measure so that we can write
$$Du=\nabla u \LL^N + D^s u.$$

The set $S_u$ of points where $u$ does not have an approximate limit is called the approximated discontinuity set, while $J_u \subseteq S_u$ is  the so called jump set of $u$ defined as the set of points $x \in
\O$ such that there exist $u^\pm(x) \in \Rb^d$ (with $u^+(x)\neq
u^-(x)$) and $\nu_u(x) \in \mathbb S^{N-1}$ satisfying
$$
\lim_{\varepsilon \to 0}\frac{1}{\varepsilon^N}\int_{\{y\in B_\e(x): (y-x)\cdot \nu_u(x)>0\}} |u(y)-u^+(x)|\,dy=0,
$$
and
$$
\lim_{\varepsilon \to 0}\frac{1}{\varepsilon^N}\int_{\{y\in B_\e(x): (y-x)\cdot \nu_u(x)<0\}} |u(y)-u^-(x)|\,dy=0.
$$

It is known that $J_u$
is a countably $\HH^{N-1}$-rectifiable Borel set.
By Federer-Vol'pert Theorem (see Theorem 3.78 in \cite{AFP}), ${\cal H}^{N-1}(S_u \setminus J_u)= 0$ for any $u \in BV(\Omega;\mathbb R^d)$.
The measure $D^s
u$ can in turn be decomposed into the sum of a jump part and a
Cantor part defined by $D^j u:=D^s u \res\, J_u$ and $D^c u:= D^s u
\res\, (\O \setminus S_u)$. We now recall the decomposition of $Du$:
$$Du= \nabla u  \LL^N + (u^+ -u^-)\otimes \nu_u {\cal H}^{N-1}\res\,
J_u + D^c u.$$
The three measures above are mutually singular. If ${\cal H}^{N-1}(B) < + \infty$, then $|D^c u|(B)=0$ and there exists a Borel set $E$ such that
$$
{\cal L}^N(E)= 0, \; |D^c u|(X)=|D^c u|(X \cap E)
$$
for all Borel sets $X \subseteq \Omega$.

\section{Relaxation}\label{decoupled}

This section is devoted to the proof of the integral representation results dealing with the decoupled models described in the introduction.

To prove Theorems \ref{thmdec1} and \ref{thmdec2} we will use the characterization for the relaxed functional of $I_W:L^1(\Omega; \mathbb R^d)\longrightarrow\mathbb{R}\cup\{+\infty\}$ defined by
\begin{equation}\label{IW}
I_W(u):= \left\{\begin{array}{ll}
\displaystyle{\int_\O W(x,u(x),\nabla u(x))\,dx }&\hbox{ if }u \in W^{1,1}(\Omega;\mathbb R^d),\vspace{0.2cm}\\
+ \infty &\hbox{otherwise}.
\end{array}
\right.
\end{equation}
%with $W$ having or not explicit dependence on $u$ according to the context of each theorem.

\noindent The relaxed functional of $I_W$ is defined by $$\displaystyle{\overline{I_W}(u):=\inf\left\{\liminf_n I_W(u_n):\ u_n\in BV(\Omega;\mathbb{R}^d),\ u_n \to u \hbox{ in }L^1\right\}}$$ and it was characterized by Fonseca-M\"{u}ller in \cite{Fonseca-Muller-relaxation}, provided (among other hypotheses) that $W$ is quasiconvex. In the next lemma we establish conditions to obtain the representation of $I_W$ in the general case, that is, with $W$ not necessarily quasiconvex.

 We will also use the following notation. The functional $I_{QW}:L^1(\Omega; \mathbb R^d)\longrightarrow\mathbb{R}\cup\{+\infty\}$ is defined by \begin{equation}\label{IQW}
I_{QW}(u):=\left\{\begin{array}{ll}
\displaystyle{\int_\O QW(x,u(x),\nabla u(x))\,dx }&\hbox{ if }u \in W^{1,1}(\Omega;\mathbb R^d),\vspace{0.2cm} \\
+ \infty &\hbox{otherwise},
\end{array}
\right.
\end{equation}
 and its relaxed functional is $$\displaystyle{\overline{I_{QW}}(u):=\inf\left\{\liminf_n I_{QW}(u_n):\ u_n\in BV(\Omega;\mathbb{R}^d),\ u_n \to u \hbox{ in }L^1\right\}}.$$

 We are now in position to establish the mentioned lemma and we notice that we make no assumptions on the quasiconvexified function $QW$.

\begin{Lemma}\label{FMnonquasiconvex}
Let $W:\Omega\times\mathbb{R}^d\times \mathbb{R}^{d\times N}\longrightarrow [0,+ \infty)$ be a continuous function and consider the functionals $I_W$ and $I_{QW}$ and their corresponding relaxed functionals defined as above. Then, if $W$ satisfies conditions (i) $\div$ (iv) of Theorem \ref{thmdec1}, the two relaxed functionals coincide in $BV(\Omega,\mathbb{R}^d)$ and moreover
$$\begin{array}{ll}
\displaystyle{\overline{I_{W}}(u)=\overline{I_{QW}}(u)=\int_\O QW(x,u(x),\nabla u(x))\,dx + \int_{J_u} \gamma(x, u^+, u^-, \nu_u)\,d {\cal H}^{N-1}+}\vspace{0.2cm}\\ \hspace{3cm}
\displaystyle{+\int_\O (QW)^\infty\left(x,u(x),\frac{d D^c u}{d|D^cu|}\right)d |D^c u|.}
\end{array}
$$
\end{Lemma}

\begin{proof} First we observe that $\overline{I_W}(u)=\overline{I_{QW}}(u)$, for every $u\in BV(\Omega;\mathbb{R}^d).$ Indeed, since $QW \leq W$, it results $\overline{I_{QW}}\leq\overline{I_W}$. Next we prove the opposite inequality in the nontrivial case that $\overline{I_{QW}}(u)<+\infty$. For fixed $\delta>0$, we can consider $u_n\in W^{1,1}(\Omega;\mathbb{R}^d)$ with $u_n \to u$ strongly in $L^1(\Omega;\mathbb R^d)$ and such that
$$
\displaystyle{\overline{I_{QW}}(u)\geq \lim_n \int_\Omega QW(x,u_n(x),\nabla u_n(x))\,dx - \delta.}
$$
Applying \cite[Theorem 9.8]{Dacorogna},
%(see also  the Relaxation Theorem in \cite[statement III.7]{Acerbi-Fusco} and \cite[Theorem 1.3]%{BFL} to have coincidence between seq. lsc envelope and $\inf \liminf$, we can either invoke %BFL or Dacorogna. I prefer to invoke Dacorogna book...the argument is self contained)
for each $n$ there exists a sequence $\{u_{n,k}\}$ converging to $u_n$ weakly in $W^{1,1}(\Omega;\mathbb R^d)$ such that
$$
\displaystyle{\int_\O QW(x, u_n(x), \nabla u_n(x))\,dx = \lim_k \int_\O W(x, u_{n,k}(x), \nabla u_{n,k}(x))\,dx.}
$$
Consequently
\begin{equation}\label{2.17BFMbend2}
\displaystyle{\overline{I_{QW}}(u)\geq \lim_n \lim_k \int_\O W(x, u_{n,k}(x), \nabla u_{n,k}(x))\,dx- \delta,}
\end{equation}
and
$$
\displaystyle{\lim_n\lim_k \| u_{n,k}- u\|_{L^1}=0.}
$$
Via a diagonal argument, there exists a sequence $\{u_{n, k_n}\}$ satisfying $u_{n,k_n} \to u$ in $L^1(\Omega;\mathbb R^d)$ and realizing the double limit in the right hand side of (\ref{2.17BFMbend2}).
Thus, it results
$$
\displaystyle{\overline{I_{QW}}(u) \geq \lim_n \int_\O W(x, u_{n ,k_n}(x), \nabla u_{n, k_n}(x))\,dx -\delta \geq \overline{I_{W}}(u)- \delta.}
$$
Letting $\delta$ go to $0$ the conclusion follows.

Finally we prove the integral representation for $\overline{I_{QW}}$ and consequently for $\overline{I_W}$. To this end we invoke \cite[Theorem 2.16]{Fonseca-Muller-relaxation} (see also \cite[Theorem 5.54]{AFP}).

By the hypotheses, and by Proposition \ref{measQW} above, $QW$ satisfies conditions (H1), (H2), (H3) and (H4) in \cite{Fonseca-Muller-relaxation}, and condition $(H5)$ follows from Proposition \ref{propperH5}. Applying \cite[Theorem 2.16]{ Fonseca-Muller-relaxation} we conclude the proof.
\end{proof}

Let $I_{QW+\varphi}:BV(\Omega;\mathbb R^d)\times L^q(\Omega;\mathbb R^m)\to \mathbb R \cup \{+\infty\}$ be the functional defined by
\begin{equation}\label{IQWvarphi}I_{QW+\varphi}(u,v):=
\left\{ \begin{array}{ll}
\displaystyle{\int_\O QW(x, u(x), \nabla u(x))\,dx + \int_\Omega \varphi(x,u(x),v(x))\,dx}\\
\\
\displaystyle{\quad \quad \quad\hbox{ if } (u,v)\in (W^{1,1}(\O;\mathbb{R}^d)\cap L^p(\Omega;\mathbb{R}^d))\times L^q(\O;\mathbb{R}^m),}\vspace{0.2cm}\\+\infty\quad\text{otherwise},\end{array}\right.
\end{equation}
and its relaxed functional as
\begin{equation}\label{IbarQWvarphi}
\begin{array}{ll}
\displaystyle{\overline{I_{QW+\varphi}}(u,v):=}\\
\\
\;\;\displaystyle{\inf\{\liminf_{n}I_{QW+\varphi}(u_n,v_n): (u_n,v_n)\in BV(\Omega;\mathbb R^d)\times L^q(\Omega;\mathbb R^m), u_n \to u \hbox{ in }L^1, v_n \rightharpoonup v \hbox{ in }L^q\}.}
\end{array}
\end{equation}
We can obtain, as in the first part of the proof of Lemma \ref{FMnonquasiconvex}, the following result.
\begin{Corollary}\label{Corollary3.2}
Let $p \geq 1$, $q \geq 1$  and let $\Omega \subset \mathbb R^N$. Assume $W:\Omega \times \mathbb R^d \times \mathbb R^{d \times N} \to [0, +\infty)$ and $\varphi:\Omega \times \mathbb R^d \times \mathbb R^m \to [0, +\infty)$ satisfying $(i) \div (iv)$ of Theorem \ref{thmdec1} and $(v)$ of Theorem \ref{thmdec1} respectively. Let $I$ and $\overline{I}$ be defined by \eqref{firstmodel} and \eqref{relaxedfunctional1} respectively. Let $I_{QW+ \varphi}$ and $\overline{I_{Qw+ \varphi}}$ be as in \eqref{IQWvarphi} and \eqref{IbarQWvarphi} respectively, then
$$
\overline{I}(u,v)=\overline{I_{QW+\varphi}}(u,v)
$$
for every $(u,v)\in BV(\Omega;\mathbb R^d)\times L^q(\Omega;\mathbb R^m)$.
\end{Corollary}

\begin{Remark}\label{Remark3.3}

We observe that, in the case $1 \leq p < +\infty$, $1 < q < \infty$, given $W:\Omega \times\mathbb R^d\times \mathbb R^{d \times N}\to \mathbb R$ and $\varphi:\Omega \times \mathbb R^d \times\mathbb R^m \to\mathbb R$, Carath\'eodory functions satisfying $(i)$ and $(v)$ of Theorem \ref{thmdec1} respectively, then, if one can provide that $C \varphi$ is still Carath\'eodory, an argument entirely similar to the first part of Lemma \ref{FMnonquasiconvex}, entails that
$$
\begin{array}{ll}
\displaystyle{\overline{I}(u,v)=\inf\left\{\liminf_{n \to + \infty} \int_\Omega (QW(x,u_n, \nabla u_n)+ C\varphi(x, u_n, v_n))\,dx: \right.} \\
\\
\;\;\;\;\;\;\;\;\;\;\;\;\;\;\;\;\;\;\;\;\;\displaystyle{\left.(u_n, v_n) \in BV(\Omega;\mathbb R^d)\times L^q(\Omega;\mathbb R^m), u_n \to u \hbox{ in }L^1, v_n \rightharpoonup v \hbox{ in } L^q \right\}}
\end{array}
$$
where $\overline{I}$ is the functional defined by \eqref{relaxedfunctional1}, $QW$ and $C\varphi$
 are defined in \eqref{qcx-rel} and \eqref{cx-rel}.
But we emphasize that since, assuming only $(v)$ of Theorem \ref{thmdec1} there may be a lack of continuity of $C\varphi(x,\cdot,\cdot)$ as observed in Remark \ref{measCvarphi}, we focus just on the relaxation of the term $\int_\Omega W(x,u,\nabla u)\,dx$ and we prove Lemma \ref{FMnonquasiconvex} (see also Corollary \ref{Corollary3.2}) in order to be allowed to assume $W$ quasiconvex without loosing generality.

\end{Remark}

We are now in position to prove Theorem \ref{thmdec1}.

%{\bf controllare bene e scrivere quali siano le ipotesi ereditate da $QW$ e $C\varphi$ per la %rimozione di $Q$ e $C$....}

%{\bf dopo la proof del teorema inserire un remark in cui si dice che $\varphi$ non si puo assumere %convex per $q=1$ anche se c'e' il risultato di CARIZA2}

\begin{proof}[Proof of Theorem \ref{thmdec1}]

The proof is divided in two parts. First we consider the case $q>1$ and then we consider $q=1$. In both cases we first prove a lower bound for the relaxed energy $\overline{I}$ and then we prove that the lower bound obtained is also an upper bound for $\overline{I}$.
\medskip

Preliminarly we observe that by virtue of Corollary \ref{Corollary3.2}, Propositions \ref{measQW}, \ref{QWinfty=}, \ref{propperH5} we can assume without loss of generality, that $W$ is quasiconvex in the last variable.
\medskip

\noindent{\sl Part 1: $q>1$.}

\noindent{\sl Lower bound.} Let $u\in BV(\O;\mathbb{R}^d)\cap L^p(\O;\mathbb{R}^d)$ and let $v\in L^q(\O;\mathbb{R}^m)$. We will prove that, for any sequences $u_n\in BV(\O;\mathbb{R}^d)$ and $v_n\in L^q(\O;\mathbb{R}^m)$ such that $u_n \to u$ in $L^1$ and $v_n \weak v$ in $L^q$,
$$\begin{array}{l}\displaystyle{\liminf_{n\to +\infty}I(u_n,v_n)\ge \int_\O W(x,u,\nabla u)\,dx + \int_{J_u}\gamma(x, u^+, u^-, \nu_u)\,d{\cal H}^{N-1}+\int_\O W^\infty\left(x,u,\frac{d D^c u}{d|D^cu|}(x)\right)d |D^c u|+}\vspace{0.2cm}\\ \hspace{3cm}\displaystyle{+\int_\O C\varphi(x,u,v)\,dx.}\end{array}$$

\noindent Let $u_n$ and $v_n$ be two sequences in the conditions described above. Then, by \cite[Theorem 2.16]{Fonseca-Muller-relaxation}
\begin{equation}\label{eq1.1.1}
\begin{array}{ll}
\displaystyle{
\int_\O W(x,u,\nabla u)\,dx + \int_{J_u}\gamma(x,u^+,u^-, \nu_u)\,d{\cal H}^{N-1}+ \int_\O W^\infty\left(x,u, \frac{d D^c u}{d|D^cu|}\right)d |D^c u|  \leq} \\
\\
\displaystyle{\quad \quad \quad \quad \leq \liminf_{n \to + \infty}\int_\O W(x,u_n,\nabla u_n)\,dx.}
\end{array}
\end{equation}

\noindent Moreover, since we can assume $\liminf_n \int_\Omega \varphi(x,u_n,v_n)\, dx < +\infty$, the bound  on $\|u_n\|_{L^p}$ provided by $(v)$, the fact that $u_n \to u$ in $L^1(\Omega)$ and consequently pointwise, guarantee that $u_n \to u$ strongly in $L^p$. Furthermore $v_n \weak v$ weakly in $L^q$ and because of the lower semi-continuity of $ C\varphi(x, \cdot, \cdot)$ (cf. \cite[Lemma 4.3]{DM-F-L-M}), it results (cf. \cite[Theorem 7.5]{Fonseca-Leoni} or \cite{Ekeland-Temam})
\begin{equation}\label{eq1.1.2}
\int_\O C \varphi(x,u,v)\,dx \leq \liminf_{n \to + \infty} \int_\O  C\varphi (x, u_n, v_n)\,dx \leq \liminf_{n \to + \infty} \int_\O \varphi(x, u_n, v_n)\,dx.
\end{equation}
\noindent Consequently, the superadditivity of the $\liminf$, gives the desired lower bound.

\medskip

\noindent{\sl Upper bound.} Let $u\in BV(\O;\mathbb{R}^d) \cap L^p(\Omega;\mathbb R^d)$ and $v\in L^q(\O;\mathbb{R}^m)$. We will prove that
\begin{equation}\label{upperboundforever}
\begin{array}{ll}
\displaystyle{\overline{I}(u,v) } &\displaystyle{\leq \int_\O W(x,u,\nabla u)\,dx + \int_{J_u}\gamma(x,u^+, u^-, \nu_u)\,d{\cal H}^{N-1}}+\vspace{0.2cm}
\\
&\displaystyle{+\int_\O W^\infty\left(x,u,\frac{d D^c u}{d|D^cu|}\right)d |D^c u|+ \int_\O C\varphi (x,u,v)\,dx}
\end{array}
\end{equation}
constructing convenient sequences $u_n\in BV(\Omega;\mathbb{R}^d)$ such that $u_n\rightarrow u$ in $L^1$ and $v_n\in L^{q}(\Omega,\mathbb{R}^m)$ such that $v_n\rightharpoonup v$ in $L^q$.

\noindent We can assume, without loss of generality, that
\begin{equation}\label{finite-assumption}
\begin{array}{ll}
\displaystyle{\int_\O W(x,u,\nabla u)\,dx + \int_{J_u}\gamma(x, u^+, u^-, \nu_u)\,d{\cal H}^{N-1}+\int_\O W^\infty\left(x,u,\frac{d D^c u}{d|D^cu|}\right)d |D^c u| +}
\\
\\
\displaystyle{ + \int_\O  C\varphi (x,u,v)\,dx<+\infty.}
\end{array}
\end{equation}
In particular, from $(v)$ it follows that $u\in L^p(\Omega;\mathbb{R}^d)$.

\noindent Moreover we suppose, without loss of generality, that $W\ge 0$ and $\varphi \ge 0$. We will consider two cases.\medskip

\noindent{\sl Case 1:} $u\in L^\infty(\O;\mathbb{R}^d)$.

Fix $M\in\mathbb{N}$. We will prove that, for some constant $c$ (independent of $M$), $$
\begin{array}{ll}
\displaystyle{\overline{I}(u,v)}&\displaystyle{\le \int_\O W(x,u,\nabla u)\,dx +\int_{J_u}\gamma(x, u^+,u^-, \nu_u)\,d{\cal H}^{N-1}+ \int_\O W^\infty\left(x,u,\frac{d D^c u}{d|D^c u|}\right)d |D^c u| +}\\
\\
&\displaystyle{+ \int_\O C \varphi (x,u,v)\,dx+\frac{c}{M}.}
\end{array}$$
Then we get the desired inequality by letting $M$ go to $+\infty$.

\noindent We proceed in three steps.\medskip

\noindent{\sl Case 1, step 1:} construction of a convenient sequence converging to $u$ in $L^1(\Omega;\mathbb{R}^d)$.

\noindent Let $\{u_n\}$ be a sequence in $W^{1,1}(\O;\mathbb{R}^d)$ such that $u_n \to u$ in $L^1$ and
$$
\begin{array}{ll}
\displaystyle{\lim \int_\O W(x,u_n,\nabla u_n)\,dx= }\\
\\=\displaystyle{\int_\O W(x,u,\nabla u)\,dx + \int_{J_u} \gamma(x,u^+,u^-, \nu_u)\,d{\cal H}^{N-1}+\int_\O W^\infty\left(x,u,\frac{d D^c u}{d|D^c u|}\right)d |D^c u|.}
\end{array}
$$
This exists by  \cite[Theorem 2.16]{Fonseca-Muller-relaxation}. Next we will truncate the sequence $u_n$.

\noindent Fix $k$ such that $e^{k}-1>2||u||_{L^\infty}$. Then, hypothesis (\ref{finite-assumption}) together with the coercivity condition of $W$ on $\xi$, cf. $(i)$, and the fact that $ \varphi \ge 0$, imply that  $\sup||\nabla u_n||_{L^1}$ is bounded by a constant independent of the sequence $u_n$. Thus
$$
\begin{array}{ll}
\displaystyle{\sum_{i=0}^{M-1}\int_{\{x\in\Omega:\ k+i\le \ln(1+|u_n|)<k+i+1\}}(1+|\nabla u_n|)\,dx} &\displaystyle{=\int_{\{x\in\Omega:\ k\le \ln(1+|u_n|)<k+M\}}(1+|\nabla u_n|)\,dx }\\
\\
&\displaystyle{\le|\O|+\sup_n||\nabla u_n||_{L^1}}
\end{array}$$ and so, for each $n\in\mathbb{N}$, we can find $i=i(n)\in\{0,...,M-1\}$ such that
\begin{equation}\label{controlling-the-gradient}
\int_{\{x\in\Omega:\ k+i\le \ln(1+|u_n|)<k+i+1\}}(1+|\nabla u_n(x)|)\,dx\le \frac{|\O|+\sup_n||\nabla u_n||_{L^1}}{M}.\end{equation}
For each $n$, and accordingly to the previous choice of $i(n)$, consider $\tau_n:\mathbb{R}_0^+\longrightarrow[0,1]$ such that $\tau_n\in C^1(\mathbb{R}_0^+)$, $|\tau_n'|\le 1$, $$\tau_n (t)=1,\ \text{if}\ 0\le t< k+i(n)\qquad\text{and}\qquad\tau_n (t)=0,\ \text{if}\ t\ge k+i(n)+1.$$
We can now define the truncated sequence. Let $g_n(z):=\tau_n(\ln(1+|z|))\,z$, and $\overline{u}_n(x)= g_n(u_n)$.
Since in a neighborhood of $0$ the function $\tau_n(\ln(1+|\cdot|))$ is identically $1$, $g_n$ is a Lipschitz, $C^1$ function with
$$
\nabla g_n(z)=\left\{
\begin{array}{l}
\tau_n(\ln(1+|z|))\,{\mathbb I}+\tau_n'(\ln(1+|z|))\frac{1}{1+|z|}\frac{z\otimes z}{|z|},\ \hbox{if }z\neq 0\vspace{0.2cm}\\ {\mathbb I},\ \text{if }z=0\end{array}\right.$$
and $|\nabla g_n(z)|\le c$. So, by Theorem 3.96 in \cite{AFP}, $\overline{u}_n\in W^{1,1}(\O; \mathbb{R}^d)$,
$\nabla\overline{u}_n=\nabla g_n(u_n)\nabla u_n\mathcal{L}^N$ and
%+(g_n(u_n^+)-g_n(u_n^-))\otimes \nu_u\mathcal{H}^{N-1}_{|J_{u_n}}+\nabla g_n(\tilde{u}_n)D^c u_n$ and
$|\nabla\overline{u}_n|\le c|\nabla u_n|$ which is bounded in $L^1$ as observed above.
%because $u_n\rightharpoonup u$ in $BV *$.
Moreover $||\overline{u}_n||_{L^\infty}\le e^{k+i(n)+1}-1\le e^{k+M}-1$ and $\overline{u}_n\rightarrow u$ in $L^1$. Indeed, if $u\equiv0$ then $||\overline{u}_n||_{L^1}\le ||u_n||_{L^1}\rightarrow 0$. If not
$$
\begin{array}{rcl}
\displaystyle{||\overline{u}_n-u||_{L^1(\O)}} & = & \displaystyle{\int_{\{x\in\O:\ 0\le \ln(1+|u_n|)<k+i(n)\}} |u_n(x)-u(x)|\,dx+}\vspace{0.2cm}\\ & &\displaystyle{+\int_{\{x\in\O:\ k+i(n)\le \ln(1+|u_n|)<k+i(n)+1\}} |\overline{u}_n(x)-u(x)|\,dx+}\vspace{0.2cm}\\ & &\displaystyle{
+\int_{\{x\in\O:\ \ln(1+|u_n|)\ge k+i(n)+1\}} |u(x)|\,dx}\vspace{0.2cm}\\& \le & \displaystyle{||u_n-u||_{L^1(\O)}+\int_{\{x\in\O:\ k+i(n)\le \ln(1+|u_n|)<k+i(n)+1\}} |\overline{u}_n(x)-u_n(x)|\,dx+} \vspace{0.2cm}\\& & +||u_n-u||_{L^1(\O)}+||u||_{L^\infty(\O)}\,\displaystyle{\left|\{x\in\O: \ln(1+|u_n|)\ge k+i(n)+1\}\right|} \vspace{0.2cm}\\& \le & \displaystyle{2||u_n-u||_{L^1(\O)}+\int_{\{x\in\O:\ k+i(n)\le \ln(1+|u_n|)<k+i(n)+1\}} |u_n(x)|\,dx+ }\vspace{0.2cm}\\& & +||u||_{L^\infty(\O)}\,\displaystyle{\left|\{x\in\O:\ \ln(1+|u_n|)\ge k+i(n)+1\}\right|} \vspace{0.2cm} \end{array}$$
these last terms converging to zero because $u_n\rightarrow u$ in $L^1$ and because of the following estimates:

$$
\begin{array}{rcl}
\displaystyle{\int_{\{x\in\O:\ k+i(n)\le \ln(1+|u_n|)<k+i(n)+1\}} |u_n(x)|\,dx} & \le & \displaystyle{\int_{\{x\in\O:\ k+i(n)\le \ln(1+|u_n|)<k+i(n)+1\}} e^{k+M}-1\,dx}\vspace{0.2cm} \\ & \le & (e^{k+M}-1) \left|\{x\in\O: |u_n|\ge e^{k+i(n)}-1\}\right| \vspace{0.2cm}\\ & \le & (e^{k+M}-1) \left|\{x\in\O:\ |u_n-u|\ge ||u||_{L^\infty(\O)}\}\right| \vspace{0.2cm}\\ & \le & \displaystyle{(e^{k+M}-1)\frac{||u_n-u||_{L^1(\O)}}{||u||_{L^\infty(\O)}}}, \end{array}$$

$$
\begin{array}{rcl}
\displaystyle{\left|\{x\in\O: \ln{(1+|u_n|)}\ge k+i(n)+1 \}\right|}& = & \displaystyle{\left|\{x\in\O: |u_n|\ge e^{k+i(n)+1}-1\}\right|} \vspace{0.2cm}\\
& \le & \displaystyle{\left|\{x\in\O: |u_n-u|\ge ||u||_{L^\infty(\O)}\}\right|}
\vspace{0.2cm}\\
& \le & \displaystyle{\frac{||u_n-u||_{L^1(\O)}}{||u||_{L^\infty(\O)}}}.
\end{array}
$$

\noindent So, we have, in particular, that $\overline{u}_n$ converges to $u$ in $L^1$ and $\overline{u}_n$ clearly belongs to $L^p(\Omega;\mathbb{R}^d)$.\medskip

\noindent{\sl Case 1, step 2:} construction of a convenient sequence $\{v_n\}$ weakly converging to $v$ in $L^q$.

We have, by $(v)$, \cite[Theorem 6.68 and Remark 6.69 $(ii)$]{Fonseca-Leoni}, for any $w\in L^1(\O;\mathbb{R}^d)$

$$\int_\O C\varphi(x,w,v)\,dx=\inf \left\{\liminf_{n \to + \infty} \int_\O \varphi(x,w,v_n)\,dx: \{v_n\}\subset L^q(\O; \mathbb R^m), v_n \weak v  \hbox{ in } L^q \right\}$$ whenever the second term is finite.

\noindent Since $q>1$ and thus $L^{q'}(\O;\mathbb{R}^m)$ is separable, we can consider a sequence $\{\psi_l\}$ of functions, dense in $L^{q'}(\O;\mathbb{R}^m)$.

\noindent Then, for each $n\in\mathbb{N}$ let $v^n_j\in L^q(\O;\mathbb{R}^m)$ be such that

$$\int_\O C \varphi(x,\overline{u}_n,v)\,dx=\lim_{j \to + \infty} \int_\O \varphi(x,\overline{u}_n,v^n_j)\,dx$$ and

$$\lim_{j \to + \infty} \int_\O (v^n_j-v)\psi_l\,dx=0,\ \forall\ l\in\mathbb{N}.$$

\noindent We then extract a diagonalizing sequence ${v_n}$ in the following way: for each $n\in\mathbb{N}$ consider $j(n)$ increasing and verifying

$$\left |\int_\O \left(\varphi(x,\overline{u}_n,v^n_{j(n)})-C \varphi (x,\overline{u}_n,v)\right)\,dx\right |\le \frac{1}{n}$$

$$\left |\int_\O (v^n_{j(n)}-v)\psi_l\,dx\right|\le \frac{1}{n},\ l=1,...,n.$$ Define then $v_n=v^n_{j(n)}$. We have $v_n$ bounded in the $L^q$ norm:
$$\int_\O |v_n|^q\,dx\le C\int_\O\varphi(x,\overline{u}_n,v_n)\,dx\le \frac{C}{n}+C\int_\O C \varphi(x,\overline{u}_n,v)\,dx\le C+C\int_\O\varphi(x,\overline{u}_n,v)\,dx$$ this last term being bounded because $\overline{u}_n$ is a bounded sequence in $L^\infty$ and because of the growth condition $(v)$ on $\varphi$.

%$\overline{u}_n$ is bounded ($|\overline{u}_n|\le e^{k+M}$) and by hypothesis %(\ref{varphicondition}).

\noindent Moreover the density of ${\psi_l}$ in $L^{q'}$ ensures that $v_n \weak v$ in $L^q$. Indeed, let $\psi\in L^q(\O;\mathbb{R}^m)$ and let $\delta>0$. Consider $l\in\mathbb{N}$ such that $||\psi_l-\psi||_{L^{q'}}\le\delta$. Then, for sufficiently large $n$, $$\left |\int_\O (v_n-v)\psi\,dx\right |\le\left |\int_\O (v_n-v)(\psi-\psi_l)\,dx\right |+\left |\int_\O (v_n-v)\psi_l\,dx\right |\le ||v_n-v||_{L^q} ||\psi_l-\psi||_{L^{q'}}+\delta\le c\delta+\delta.$$

\noindent{\sl Case 1, step 3:} upper bound for $\overline{I}$.

%We have, by \cite[Theorem 6.68]{Fonseca-Leoni}, for any $w\in L^1(\O;\mathbb{R}^d)$

%$$\int_\O \varphi (x,w,v)\,dx=\inf \left\{\liminf_{n \to + \infty} \int_\O \varphi(x,w,v_n)\,dx: %\{v_n\}\subset L^q(\O; \mathbb R^m), v_n \weak v  \hbox{ in } L^q \right\}$$ whenever the %second term is finite.

\noindent Start  remarking that

$$\limsup_{n \to + \infty} \int_\O \varphi(x,\overline{u}_n,v_n)\,dx\le\int_\O  C\varphi (x,u,v)\,dx.$$ Indeed,
$$\begin{array}{rcl}\displaystyle{\limsup_{n \to + \infty} \int_\O \varphi(x,\overline{u}_n,v_n)\,dx}&=&\displaystyle{\limsup_{n \to + \infty} \int_\O \left(\varphi(x,\overline{u}_n,v_n)-C \varphi(x,\overline{u}_n,v)+C \varphi(x,\overline{u}_n,v)\right)\,dx}\vspace{0.2cm}\\& \le& \displaystyle{\limsup_{n \to + \infty}\left( \frac{1}{n}+ \int_\O C \varphi(x,\overline{u}_n,v)\,dx	\right).}\end{array}$$

\noindent As observed in Remark \ref{measCvarphi}, $ C\varphi(x,\cdot,v)$ is upper semi-continuous. By the pointwise convergence of $\overline{u}_n$ towards $u$ (up to a subsequence), we have
$$\limsup_{n \to + \infty}  C\varphi(x,\overline{u}_n,v)\le  C\varphi(x,u,v).$$ Moreover the fact that $\overline{u}_n$ is bounded in $L^\infty$ and the  hypothesis $(v)$ allows to apply the ``inverted'' Fatou's lemma and get the desired inequality.

\noindent Now we have
$$\begin{array}{rcl}\displaystyle{\int_\O W(x,\overline{u}_n,\nabla \overline{u}_n)\,dx} & = & \displaystyle{\int_{\{x\in\O:\ 0\le\ln(1+|u_n|)<k+i(n)\}} W(x,u_n,\nabla u_n)\,dx+}\vspace{0.2cm}\\ &&\displaystyle{+\int_{\{x\in\O:\ k+i(n)\le \ln(1+|u_n|)<k+i(n)+1\}} W(x,\overline{u}_n,\nabla \overline{u}_n)\,dx+}\vspace{0.2cm}
\\ &&\displaystyle{+\int_{\{x\in\O:\ \ln(1+|u_n|)\ge k+i(n)+1\}} W(x,0,0)\,dx}\vspace{0.2cm}
\\& \le & \displaystyle{\int_\O W(x,u_n,\nabla u_n)\,dx+\int_{\{x\in\O:\ k+i(n)\le \ln(1+|u_n|)<k+i(n)+1\}} C(1+|\nabla \overline{u}_n|)\,dx+ }\vspace{0.2cm}\\& & +C\,\displaystyle{\left|\{x\in\O:\ \ln(1+|u_n|)\ge k+i(n)+1\}\right|} \vspace{0.2cm}\end{array}$$
(where it has been used the growth condition (i). Using the expression of $\overline{u}_n$, by \cite[Theorem 3.96]{AFP}, we have $|\nabla \overline{u}_n|\le c|\nabla u_n|$ and so, using (\ref{controlling-the-gradient}), we get
$$\limsup_{n \to + \infty}\int_{\{x\in\O:\ k+i(n)\le \ln(1+|u_n|)<k+i(n)+1\}} C(1+|\nabla \overline{u}_n|)\,dx\le c\,\frac{|\O|+\sup||\nabla u_n||_{L^1}}{M}=\frac{c}{M}$$ (note that $c$ is independent of $n$ and of the sequence $u_n$, and it doesn't represent always the same constant).

\noindent Moreover, since $\displaystyle{\left|\{x\in\O:\ \ln(1+|u_n|)\ge k+i(n)+1\}\right|}\rightarrow 0$ as $n\rightarrow +\infty$ (as already seen in the case where $||u||_{L^\infty}\neq 0$) we get,
$$\limsup_{n \to + \infty}\int_\O W(x,\overline{u}_n,\nabla \overline{u}_n)\,dx\le \lim_{n \to + \infty}\int_\O W(x,u_n,\nabla u_n)\,dx+\frac{c}{M}.$$
Note that if $u=0$ we can still get $\displaystyle{\left|\{x\in\O:\ \ln(1+|u_n|)\ge k+i(n)+1\}\right|}\rightarrow 0$:

$$\displaystyle{\left|\{x\in\O:\ \ln(1+|u_n|)\ge k+i(n)+1\}\right|} \le \displaystyle{\left|\{x\in\O:\ |u_n|\ge e^{k+1}-1\}\right|} \le \frac{||u_n||_{L^1}}{e^{k+1}-1}\rightarrow 0$$ since $u_n\rightarrow 0$ in $L^1$.

\noindent Finally, we get, as desired,
$$\begin{array}{rcl}\overline{I}(u,v)&\le& \displaystyle{\liminf_{n \to + \infty} \int_\O W(x,\overline{u}_n,\nabla \overline{u}_n)\,dx + \int_\O \varphi(x,\overline{u}_n,v_n)\,dx}\vspace{0.2cm}\\& \le & \displaystyle{\limsup_{n \to + \infty} \int_\O W(x,\overline{u}_n,\nabla \overline{u}_n)\,dx + \limsup_{n \to + \infty} \int_\O \varphi(x,\overline{u}_n,v_n)\,dx}\vspace{0.2cm}\\& \le & \displaystyle{\int_\O W(x, u, \nabla u)\,dx +
\int_{J_u}\gamma(x,u^+,u^-, \nu_u)\,d{\cal H}^{N-1}+ \int_\O W^\infty\left(x,u,\frac{dD^c u}{d|D^cu|}\right)d |D^c u|+}\vspace{0.2cm}\\
& \quad &\displaystyle{+\frac{c}{M}+\int_\O  C\varphi(x,u,v)\,dx.}\vspace{0.2cm}
\end{array}$$\smallskip

\noindent{\sl Case 2:} arbitrary $u\in BV(\O;\mathbb{R}^d)\cap L^p(\Omega;\mathbb{R}^d)$.

\noindent To achieve the upper bound on this case, we will reduce ourselves to Case 1 by means of a truncature argument developed in \cite[Theorem 2.16, Step 4]{Fonseca-Muller-relaxation}, in turn inspired by \cite[Theorem 4.9]{Ambrosio-Mortola-Tortorelli}. We reproduce the same argument as in \cite{Fonseca-Muller-relaxation} for the reader's convenience.

\noindent Let $\phi_n \in C^1_0(\mathbb R^d;\mathbb R^d)$ be such that
$$
\phi_n(y)=y \hbox{ if } y \in B_n(0), \;\; \|\nabla \phi_n\|_{L^\infty}\leq 1,
$$
and fix $u \in BV(\Omega;\mathbb R^d)\cap L^p(\Omega;\mathbb R^d)$.

%\noindent The chosen norm  for matrices is
%$$
%\|A\|:= \sup\{|Ax|: |x|\leq 1\},
%$$
%and that guarantees that $\|\nabla \phi_n\|_{L^\infty}\leq 1$.
As proven in \cite[Theorem 4.9]{Ambrosio-Mortola-Tortorelli}, directly from the definitions and properties for the approximate discontinuity set and the triplets $(u^+, u^-, \nu_u)$ (see Subection 2.2), it results that
\noindent
$$
\begin{array}{ll}
 J_{\phi_n(u)}\subset J_u,\\
\\
(\phi_n(u)^+, \phi_n(u)^-, \nu_{\phi_n(u)})=(\phi_n(u^+), \phi_n(u^-), \nu_u) \hbox{ in } J_{\phi_n(u)}.
\end{array}
$$
Moreover one has
\begin{equation}\label{5.23FM}
|D\phi_n(u)|(B) \leq |D(u)|(B), \hbox{ for every Borel set } B\subset \Omega.
\end{equation}
Consequently
$$
\phi_n(u) \in BV(\Omega;\mathbb R^d)\cap L^\infty(\Omega;\mathbb R^d).
$$
\noindent Since $\phi_n(u) \to u$ in $L^1$, by the lower semicontinuity of $\overline{I}$ (since $q>1$)  and by Case 1 we get
$$
\begin{array}{ll}
\displaystyle{\overline{I}(u,v)\leq \liminf_{n \to + \infty}\left[\int_\Omega W(x, \phi_n(u), \nabla \phi_n(u))\,dx +\int_{J_{\phi_n(u)}}\gamma (x, \phi_n(u)^+, \phi_n(u)^-, \nu_{\phi_n(u)})\,d {\cal H}^{N-1}+\right.}\\
\\
\displaystyle{\left.\;\;\;\;\quad\quad+\int_\Omega W^\infty\left(x, \phi_n(u),\frac{d D^c(\phi_n(u))}{d |D^c(\phi_n(u))|}\right)d |D^c\phi_n(u)|+\int_\Omega C\varphi(x, \phi_n(u), v)\,dx\right].}
\end{array}
$$

\noindent By the upper semicontinuity of $\gamma$ in all of its arguments as stated in \cite[(c) of Lemma 2.15]{Fonseca-Muller-relaxation} and by the fact that $\gamma(x,a,b,\nu)\leq C|a-b|$ for every $(x,a,b,\nu) \in \Omega \times \mathbb R^d \times \mathbb R^d \times S^{N-1}$ (see \cite[(d) Lemma 2.15]{Fonseca-Muller-relaxation}) and the properties of $\phi_n$ we have
$$
\gamma(x, \phi_n(u^+), \phi_n(u^-),\nu_u) \leq C|u^+ - u^-|,
$$
and so, by Fatou's Lemma we obtain
$$
\displaystyle{\limsup_{n\to +\infty} \int_{J_{\phi_n(u)}} \gamma(x, \phi_n(u)^+, \phi_n(u)^-, \nu_u)\,d {\cal H}^{N-1}\leq \int_{J_u} \gamma(x, u^+, u^-, \nu_u)\,d{\cal H}^{N-1}.}
$$

\noindent Moreover  we have
\begin{equation}\label{newterm}
\limsup_{n \to +\infty} \int_\Omega C\varphi(x, \phi_n(u),v)\,dx= \int_\Omega C\varphi(x,u,v)\,dx
\end{equation}
Indeed, as already observed in step 2, $C\varphi(x,\cdot, v)$ is upper semicontinuous and $\phi_n(u)$ is pointwise converging to $u$ and thus  we can apply the inverted Fatou's lemma.

\noindent For what concerns the other terms, setting $\Omega_n:= \left\{x \in \Omega\setminus J_u: |u(x)|\leq n\right\}$, we have
$$
\begin{array}{ll}
\displaystyle{\limsup_{n\to +\infty}\int_\Omega W(x,\phi_n(u), \nabla \phi_n(u))\,dx=}\\
\\
\displaystyle{=\limsup_{n\to +\infty}\left[\int_{\Omega_n}W(x, \phi_n(u), \nabla \phi_n(u))\,dx + \int_{(\Omega \setminus \Omega_n)\setminus J_u}W(x, \phi_n(u),\nabla \phi_n(u))\,dx \right]}\\
\\
\displaystyle{\leq \int_\Omega W(x, u, \nabla u)\,dx + \limsup_{n \to +\infty}C\left[|\Omega \setminus \Omega_n| + |D\phi_n(u)|((\Omega \setminus \Omega_n)\setminus J_u)\right].}
\end{array}
$$
On the other hand by \eqref{5.23FM} we deduce that
$$
\limsup_{n \to +\infty}|D\phi_n(u)|((\Omega \setminus \Omega_n)\setminus J_u)\leq \limsup_{n \to +\infty}|Du|(\Omega \setminus(\Omega_n \cup J_u))=0
$$
and so
$$
\limsup_{n \to +\infty} \int_\Omega W(x, \phi_n(u), \nabla \phi_n(u))\,dx \leq \int_\Omega W(x, u, \nabla u)\,dx.
$$
Similarly
$$
\begin{array}{ll}
\displaystyle{\limsup_{n \to +\infty}\int_\Omega W^\infty\left(x, \phi_n(u), \frac{d D^c \phi_n(u)}{d |D^c \phi_n(u)|}\right)d |D^c \phi_n(u)|=}\\
\\
\displaystyle{=\limsup_{n \to +\infty}\int_{\Omega_n}W^\infty\left(x, \phi_n(u), \frac{d D^c \phi_n(u)}{d |D^c \phi_n(u)|}\right) d |D^c \phi_n(u)|+}\\
\\
\quad\displaystyle{+\limsup_{n \to +\infty}\int_{(\Omega \setminus \Omega_n) \setminus J_u}W^\infty\left(x, \phi_n (u), \frac{d D^c\phi_n(u)}{d |D^c \phi_n(u)|}\right)d |D^c \phi_n(u)|\leq} \\
\\
\displaystyle{\leq \int_\Omega W^\infty\left(x ,u, \frac{d D^c u}{d |D^c u|}\right) d |D^c u|+ C \limsup_{n \to +\infty} [|D\phi_n (u)|((\Omega \setminus \Omega_n)\setminus J_u) ] =}\\
\\
\displaystyle{=\int_\Omega W^\infty\left(x, u,\frac{d D^c u}{d |D^c u|}\right)d |D^c u|.}
\end{array}
$$
This finishes the proof.

\medskip

\noindent {\sl Part 2: $q=1$.}

{\sl Lower bound.} Let $u \in BV(\Omega;\mathbb R^d) \cap L^p(\Omega;\mathbb R^d)$, $v \in L^1(\Omega;\mathbb R^m)$, $u_n \in BV(\Omega;\mathbb R^d)$ and $v_n \in L^1(\Omega;\mathbb R^m)$ such $u_n \to u$ strongly in $L^1$ and $v_n \rightharpoonup v$ in $L^1$. Then by Lemma \ref{FMnonquasiconvex} exactly as in the case $q>1$, \eqref{eq1.1.1} continues to hold. Moreover \cite[Theorem 1.1]{BFL}, ensures that
$$
\int_\Omega C \varphi(x, u, v)\,dx \leq \liminf_{n \to + \infty} \int_\Omega \varphi(x, u_n, v_n)\,dx.
$$
Again the lower bound follows from the superadditivity of the liminf.

{\sl Upper bound.}  Let $u \in BV(\Omega;\mathbb R^d)\cap L^p(\Omega;\mathbb R^d)$ and $v \in L^1(\Omega;\mathbb R^m)$. We aim to prove \eqref{upperboundforever},  constructing convenient sequences $u_n \in BV(\Omega;\mathbb R^d)$ and $v_n \in L^1(\Omega;\mathbb R^m)$ with $u_n \to u$ in $L^1$ and $v_n \rightharpoonup v$ in $L^1$.

\noindent{\sl Case 1.} As in the case $q>1$ we first assume that $u \in L^\infty(\Omega;\mathbb R^d)$ and develop our proof in three steps.

\noindent{\sl Case 1, step 1.} The step 1 is identical to Case 1, step 1 proven for $q>1$.

\noindent{\sl Case 1, step 2.} For what concerns  this step, we preliminarly consider a continuous increasing function $\theta:[0, + \infty) \to [0, + \infty)$ such that
\begin{equation}\label{uperlineargrowth}
\displaystyle{\lim_{t \to +\infty} \frac{\theta(t)}{t}= +\infty}.
\end{equation}
Then consider a decreasing sequence $\varepsilon \to 0$ and take the functional $I_{\varepsilon}:BV(\Omega;\mathbb R^d)\times L^1(\Omega;\mathbb R^m)\to \mathbb R \cup \{+\infty\}$, defined as
\begin{equation}\label{Iepsilon}
I_\varepsilon(u,v):= I(u,v)+ \varepsilon \int_\Omega \theta(|v|)\,dx.
\end{equation}
Let $C (\varphi(x,u,\cdot) + \varepsilon \theta(|\cdot|))$ be the convexification of $\varphi(x,u,\cdot) + \varepsilon \theta(|\cdot|)$ as in \eqref{cx-rel}.

By \cite[Theorem 6.68 and Remark 6.69]{Fonseca-Leoni}, we have that for every $w \in L^1(\Omega;\mathbb R^m)$
$$
\int_\Omega C(\varphi(x,w,v) + \varepsilon \theta(|v|))\,dx =\inf\left\{\liminf_{n \to +\infty}\int_\Omega(\varphi(x,w,v_n)+ \varepsilon \theta(|v_n|))\,dx: v_n \rightharpoonup v \hbox { in }L^1\right\}
$$
whenever the second term is finite.  Moreover the left hand side coincides with the sequentially weakly-$L^1$ lower semicontinuous envelope. Consequently for every $n \in \mathbb N$, let $\overline{u}_n$ be the sequence constructed in {\sl Case 1, step 1} and let $v_j^n \in L^1(\Omega;\mathbb R^m)$ be such that $v_j^n \rightharpoonup v \hbox{ in }L^1 \hbox{ as } j \to +\infty$ and
$$\int_\Omega C(\varphi(x,\overline{u}_n, v)+ \varepsilon \theta(|v|))\,dx=\lim_{j \to +\infty}\int_\Omega(\varphi(x, \overline{u}_n, v_j^n)+ \varepsilon \theta(|v_j^n|))\,dx.$$
The proof now develops as in \cite[Proposition 3.18]{Fonseca-Leoni}. The growth condition $(v)$ and the fact that $\overline{u_n}$ is bounded in $L^\infty$ and thus in $L^1$, entails that there exists a constant $M$ such that
\begin{equation}\label{3.11FLbook}
\displaystyle{\sup_{n ,j \in \mathbb N} \int_\Omega \theta(|v_j^n|)\,dx \leq M.}
\end{equation}
We observe that the growth conditions on $\theta$ guarantee that $\sup_{n, j \in \mathbb N}\|v_j^n\|_{L^1(\Omega)} \leq C(M)$. Moreover the separability of $C_0(\Omega)$ allows us to consider a dense sequence of functions $\{\psi_l\}$.

Next, mimicking the argument used in the analogous step for $q>1$, for every $\varepsilon >0$ we construct a diagonalizing sequence $v_n$ as follows.
For each $n \in \mathbb N$ consider $j(n)$ increasing and such that
$$
\begin{array}{ll}
\displaystyle{\left|\int_\Omega \left(\varphi(x, \overline{u}_n, v^n_{j(n)}) + \varepsilon \theta(|v^n_{j(n)}|) - C(\varphi(x, \overline{u}_n, v) + \varepsilon \theta(|v|))\right) dx \right| \leq \frac{1}{n}}\\
\\
\displaystyle{\left|\int_\Omega (v^n_{j(n)}- v)\psi_l\, dx \right| \leq \frac{1}{n},\, l= 1,\dots, n.}
\end{array}
$$
Define $v_n:= v^n_{j(n)}$. The bounds on $\theta$, the fact that $\overline{u}_n$ is bounded in $L^1$ and the separability of $C_0(\Omega)$ guarantee that $v_n \weakstar v$ in ${\cal M}(\Omega)$ and moreover, \eqref{3.11FLbook}, Dunford-Pettis' theorem entail that the convergence of $v_n$ towards $v$ is weak-$L^1$.

\noindent{\sl Case 1, step 3.} Arguing as in the first part of {\sl Case 1, step 3} for $q>1$, we can prove that
$$
\displaystyle{\limsup_{n \to +\infty} \int_\Omega \left(\varphi(x, \overline{u}_n, v_n) +\varepsilon \theta(|v_n|)\right)\,dx \leq \int_\Omega C(\varphi(x,u,v)+ \varepsilon \theta(|v|))\,dx.}
$$
Next we define
\begin{equation}\label{Iepsilonbar}
\displaystyle{\overline{I}_\varepsilon(u,v):=\inf\{\liminf_{n \to +\infty} I_\varepsilon(u_n, v_n): (u_n, v_n) \in BV(\Omega;\mathbb R^d)\times L^1(\Omega;\mathbb R^m), u_n \to u \hbox{ in }L^1, v_n \rightharpoonup v \hbox{ in }L^1\}.}
\end{equation}
The same argument of the last part in {\sl Case 1, step 3}, for $q>1$, allows to prove that
\begin{equation}\label{Lebmonotoneupperb}
\begin{array}{ll}
\displaystyle{\overline{I}_\varepsilon(u,v) }& \displaystyle{\leq \int_\Omega W(x,u,\nabla u)\,dx + \int_{J_u} \gamma(x, u^+, u^-, \nu_u)\,d{\cal H}^{N-1} +}\vspace{0.2cm}
\\
&\quad\displaystyle{+\int_\Omega W^\infty\left(x, u, \frac{d D^c u}{d|D^c u|}\right)d |D^c u|+ \int_\Omega C\left(\varphi(x, u,v)+ \varepsilon \theta(|v|)\right)\,dx}
\end{array}
\end{equation}
for every $u \in BV(\Omega;\mathbb R^d) \cap L^\infty(\Omega;\mathbb R^d)$ and $v \in L^1(\Omega;\mathbb R^m)$.
On the other hand we observe that the sequence $\overline{I}_\varepsilon(u,v)$ is increasing in $\varepsilon$  and $\overline{I} \leq \overline{I}_\varepsilon$ for every $\varepsilon$.
Moreover by virtue of the increasing behaviour in $\varepsilon$ of $\varphi + \varepsilon \theta$, invoking \cite[Proposition 4.100]{Fonseca-Leoni} it results that for every $(x,u)\in \Omega \times \mathbb R^d$, we have
$$
\displaystyle{ \inf_\varepsilon C\left(\varphi(x,u,v)+ \varepsilon \theta(|v|) \right) = \lim_{\varepsilon \to 0} C\left(\varphi(x,u,v)+ \varepsilon \theta(|v|) \right) = C \varphi(x,u,v) .}
$$
Thus applying Lebesgue monotone convergence theorem we have
\begin{equation}\label{3.15bis}
\begin{array}{ll}
\displaystyle{\overline{I}(u,v)\leq \lim_{\varepsilon \to 0}\overline{I}_\varepsilon(u,v)= \lim_{\varepsilon \to 0}\left(\int_\Omega W(x,u,\nabla u)\,dx +\int_{J_u}\gamma(x,u^+,u^-, \nu_u)\,d {\cal H}^{N-1}+ \right.}
\\
\\
\displaystyle{+\left.\int_\Omega W^\infty\left(x,u,\frac{d D^c u}{d |D^c u|}\right)d |D^c u|+
\int_\Omega C\left(\varphi(x,u,v)+ \varepsilon \theta(|v|)\right)\,dx\right)=}\\
\\
\displaystyle{ =\int_\Omega W(x,u,\nabla u)\,dx + \int_{J_u}\gamma(x, u^+, u^-, \nu_u)d {\cal H}^{N-1} + \int_\Omega W^\infty\left(x, u,\frac{d D^c u}{d |D^c u|}\right)d |D^c u|+ \int_\Omega C\varphi(x,u,v)\,dx,}
\end{array}
\end{equation}
for every $(u,v)\in (BV(\Omega;\mathbb R^d)\cap L^\infty(\Omega;\mathbb R^d))\times L^1(\Omega;\mathbb R^m)$.

\medskip

\noindent{\sl Case 2.} Now we consider $u \in BV(\Omega;\mathbb R^d)\cap L^p(\Omega;\mathbb R^d)$ and $v \in L^1(\Omega;\mathbb R^m)$.

To achieve the upper bound we can preliminarly observe that, a proof entirely similar to \cite[Proposition 3.18]{Fonseca-Leoni}, guarantees that for every $\varepsilon >0$, the functional $\overline{I}_\varepsilon(u,v)$, defined in \eqref{Iepsilonbar} is sequentially weakly lower semicontinuous with respect to the topology $L^1(\Omega;\mathbb R^d)_{\rm strong}\times L^1(\Omega;\mathbb R^m)_{\rm weak}$.
Thus, arguing exactly as in the {\sl Case 2}, for $q>1$, we have that
\begin{equation}\label{upperboundIepsilon}
\begin{array}{ll}
\displaystyle{\overline{I}_\varepsilon(u,v)} &\displaystyle{\leq \int_\Omega W(x,u,\nabla u)\,dx + \int_{J_u}\gamma(x,u^+, u^-, \nu_u)\,d{\cal H}^{N-1}+ \int_\Omega W^\infty\left(x,u, \frac{d D^c u}{d |D^c u|}\right)d |D^c u| +}\\
\\
&\displaystyle{+\int_\Omega C(\varphi(x,u, v) + \varepsilon \theta(|v|))\,dx.}
\end{array}
\end{equation}

Finally the monotonicity argument for $\varepsilon$ invoked in the {\sl Case 1, step 3} for $q=1$ can be recalled also in this context leading to the same inequality in \eqref{3.15bis} for every $u \in BV(\Omega;\mathbb R^d)\cap L^p(\Omega;\mathbb R^d)$ and for every $v \in L^1(\Omega;\mathbb R^m)$, and that concludes the proof of \eqref{upperboundforever}.
\end{proof}

Now we present the proof of Theorem \ref{thmdec2}, which is much easier than the latter one, since, by virtue of the continuous embedding of $BV(\Omega;\mathbb R^d)$ in $L^\frac{N}{N-1}(\Omega;\mathbb R^d)$, it does not involve any truncature argument.

\begin{proof}[Proof of Theorem \ref{thmdec2}]
We omit the details of the proof since it develops in the same way as that of Theorem \ref{thmdec1}.
First we invoke Corollary \ref{Corollary3.2} and assume without loss of generality that $W$ is quasiconvex in the last variable. Then we prove a lower bound for the relaxed energy and finally we show that the lower bound is also an upper bound.
\medskip
As in Theorem \ref{thmdec1} we may consider two separate cases: $q>1$ and $q=1$.

\noindent {\sl Lower bound for the cases $q=1$ and $q>1$.} The proof of the lower bound is identical to that of Theorem \ref{thmdec1}.

%Indeed, despite we consider now an explicit dependence on $u$ in the function $W$, the %hypotheses allow to apply Lemma \ref{FMnonquasiconvex} and get the equivalent inequality to %(\ref{eq1.1.1}).
\medskip

\noindent{\sl Upper bound, case $q>1$.} Let $u\in BV(\O;\mathbb{R}^d)$ and $v\in L^q(\O;\mathbb{R}^m)$. We can assume \begin{equation}\label{finite-assumptiondec}
\begin{array}{ll}
\displaystyle{\int_\O W(x,u,\nabla u)\,dx + \int_{J_u} \gamma(x, u^+, u^-, \nu_u)\,d {\cal H}^{N-1}+\int_\O W^\infty\left(x,u,\frac{d D^c u}{d|D^cu|}\right)d |D^c u| + }\vspace{0.2cm}\\
\displaystyle{+\int_\O C \varphi(x,u,v)\,dx<+\infty.}
\end{array}
\end{equation}

Without loss of generality, we assume also that $W$ and $\varphi\ge 0$. Applying  Lemma \cite[Theorem 2.16]{Fonseca-Muller-relaxation}, we can get a sequence $\{u_n\}$ in $W^{1,1}(\O;\mathbb{R}^d)$ such that $u_n \to u$ in $L^1$ and
$$
\begin{array}{ll}
\displaystyle{\lim \int_\O W(x,u_n,\nabla u_n)\,dx= \int_\O W(x,u,\nabla u)\,dx + \int_{J_u} \gamma(x,u^+,u^-,\nu_u)\,d {\cal H}^{N-1}+}\vspace{0.2cm}\\
\hspace{4.2cm}
\displaystyle{+\int_\O W^\infty\left(x,u, \frac{d D^c u}{d|D^cu|}\right)d |D^c u|.}
\end{array}
$$

We observe that, by the coercivity condition on $W$ and by (\ref{finite-assumptiondec}), $\nabla u_n$ is bounded in $L^1$. Moreover, the continuous embedding of $BV(\Omega;\mathbb{R}^d)$ in $L^{\frac{N}{N-1}}(\Omega;\mathbb{R}^d)$, imply that $u_n$ is bounded in $L^\frac{N}{N-1}(\Omega;\mathbb{R}^d)$ and thus in  $L^p(\Omega;\mathbb{R}^d)$ since we are assuming $1\le p\le \frac{N}{N-1}$.

Then, as in the proof of Theorem \ref{thmdec1}, {\sl Case 1, step 2}, $q>1$ we can construct a recovery sequence $v_n$ using the relaxation theorem  \cite[Theorem 6.68]{Fonseca-Leoni} and the same diagonalizing argument.
We emphasize that there is no need to make a preliminary truncature of the recovery sequence $u_n$.
Indeed, to ensure that $v_n$ is bounded in $L^q(\Omega;\mathbb{R}^m)$ (required to obtain the weak convergence of $v_n$ towards $v$ in $L^q$) it suffices to use the growth condition of $\varphi$ and the fact that $u_n$ is bounded in $L^p$.

Therefore
it is possible to get $v_n \weak v$ in $L^q$ and such that
$$\limsup_{n \to + \infty} \int_\O \varphi(x,u_n,v_n)\,dx\le\int_\O C \varphi(x,u,v)\,dx.$$
 The upper bound then follows by the sub-additivity of the limsup.

\medskip
\noindent{Upper bound, case $q=1$.} In analogy with the case $q>1$ there is no need of truncature because of the continuous embedding of $BV$ in $L^{\frac{N}{N-1}}$. As for Theorem \ref{thmdec1} it suffices to approximate the functional $I$ by $I_\varepsilon$ in \eqref{Iepsilon} and consequently it is enough to use, for the correspective relaxed functional, the diagonalization argument adopted in Theorem \ref{thmdec1}, Case 1, step 2 for $q=1$ via an application of Dunford-Pettis' theorem. Finally the monotonicity behaviour in $\varepsilon$ of $\overline{I}_\varepsilon$, the approximation of the energy densities allowed by \cite[Proposition 4.100]{Fonseca-Leoni} and the Lebegue monotone convergence theorem conclude the proof.
\end{proof}
\bigskip

{\bf Acknowledgements} The authors would like to thank Irene Fonseca and Giovanni Leoni for their invaluable comments on the subject of this paper and acknowledge the  Center for Nonlinear Analysis, Carnegie Mellon University, Pittsburgh (PA) for the kind hospitality and support.

A.M. Ribeiro thanks Universit\`{a} di Salerno and E. Zappale thanks Universidade Nova de Lisboa for their support and hospitality.

The research of A. M. Ribeiro was partially supported by Funda\c c\~ao para a Ci\^{e}ncia e Tecnologia (Portuguese Foundation for Science and Technology) through  the Pos-Doc Grant SFRH/BPD/34477/2006, Financiamento Base 2010-ISFL/1/297 from FCT/MCTES/PT. The research of both authors has been partially supported by  Funda\c c\~ao para a Ci\^{e}ncia e Tecnologia (Portuguese Foundation for Science and Technology) through
UTA-CMU/MAT/0005/2009 and PTDC/MAT/109973/2009.


\begin{thebibliography}{4}
%\bibitem{Acerbi-Fusco}{\sc E. Acerbi \& N. Fusco}, {\it Semicontinuity problems in the Calculus of %variations}, {\rm Arch. Rational Mech. Anal.}, {\bf 86}, (1984), 125-145.
%\bibitem{Alberti}{\sc G. Alberti}, {\it Rank one property for derivatives of functions with bounded %variation.}, {\rm Proc. R. Soc. Edinb.}, Sect. A {\bf 123}, No. 2, (1993), 239-274.
\bibitem{Ambrosio-Dal Maso}{\sc L. Ambrosio \& G. Dal Maso}, {\it On the Relaxation in $BV(\Omega;\mathbb R^m)$ of Quasi-convex Integrals}, {\rm Journal of Functional Analysis}, {\bf 109}, (1992), 76-97.
\bibitem{AFP}{\sc L. Ambrosio, N. Fusco \& D. Pallara}, {\it Functions of Bounded Variation and Free Discontinuity Problems}, {\rm Clarendon Press, Oxford}, (2000).
\bibitem{Ambrosio-Mortola-Tortorelli}{\sc L. Ambrosio, S. Mortola \& V. M. Tortorelli}, {\it Functionals with linear growth defined on vector valued BV functions}, {\rm J. Math. Pures Appl.}, IX. S\'er. 70, No. 3, (1991), 269-323.
\bibitem{AAB-FC1}{\sc J.-F. Aujol, G. Aubert, L. Blanc-F\'eraud \& A. Chambolle}, {\it Image decomposition
into a bounded variation component and an oscillating component}, {\rm Journal of
Mathematical Imaging and Vision}, {\bf 22} (1),  (2005), 71-88.
\bibitem{AAB-FC2}{\sc J.-F. Aujol, G. Aubert, L. Blanc-F\'eraud \& A. Chambolle}, {\it Decomposing an image:
Application to SAR images}, {\rm Scale-Space '03}, Vol. {\bf 2695 }of Lecture Notes in
Computer Science, (2003), 297-312.
\bibitem{AK}{\sc J.-F. Aujol \& S.H. Kang}, {\it Color image decomposition and restoration}, {\rm Journal of Visual Communication and Image Representation}, {\bf 17}, No. 4,  (2006), 916-928.
\bibitem{BZZ}{\sc J. F. Babadjian, E. Zappale \& H. Zorgati}, {\it Dimensional reduction for energies with linear growth involving the bending moment}, {\rm J. Math. Pures Appl. (9)}, {\bf 90}, No. 6, (2008), 520-549.
%\bibitem{BFMbend}{\sc G. Bouchitte, I. Fonseca \& M.L. Mascarenhas }, {\it The Cosserat vector %in membrane theory: a variational approach}, {\rm J. Convex Anal.}, {\bf 16}, No. 2, (2009), %351-365.
\bibitem{BFL}{\sc A. Braides, I. Fonseca \& G. Leoni}, {\it ${\cal A}$-quasiconvexity: relaxation and homogenization}, {\rm ESAIM: Control, Optimization and Calculus of Variations},
{\bf 5}, (2000), 539-577.
%\bibitem{B}{\sc G. Buttazzo}, {\rm Semicontinuity, relaxation and integral representation in the %calculus of variations.}, {\rm Pitman Research Notes in Mathematics Series}, {\bf 207}, Harlow: %Longman Scientific \& Technical; New York: John Wiley \& Sons. 222 p.,  (1989).
\bibitem{CRZ}{\sc G. Carita, A.M. Ribeiro \& E. Zappale}, {\it An homogenization result in $W^{1,p}\times L^q$}, {\rm J. Convex Anal.}, {\bf 18}, No. 4,  (2011), 1093--1126.
\bibitem{CRZ2}{\sc G. Carita, A.M. Ribeiro \& E. Zappale}, {\it Relaxation for some
integral functionals in $W^{1,p}_w \times L^q_w$},  {\rm Bol. Soc. Port. Mat.}, {\bf 2010}, Special Issue, 47--53.
\bibitem{Dacorogna}{\sc B. Dacorogna}, {\it Direct Methods in the Calculus of Variations}, Second Edition, {\rm Applied Mathematical Sciences}, {\bf 78}, {\rm Berlin: Springer}, {\bf  xii}, (2008), 619 p.
%\bibitem{Dal Maso-Modica}{\sc G. Dal Maso \& L. Modica},{\it A general theory of variational %functionals}, Topics in Functional Analysis, (1980-81),{\rm Quaderni Scuola Norm. Sup., Pisa}, %(1981), 149-221.
\bibitem{DM-F-L-M}{\sc G. Dal Maso, I. Fonseca, G. Leoni \& M. Morini}, {\it Higher order quasiconvexity reduces to quasiconvexity}, {\rm Arch. Rational Mech. Anal}, {\bf 171}, No. 51, (2004), 55-81.
%\bibitem{DeGiorgi-Letta}{\sc E. De Giorgi \& G. Letta}, {\it Une notion g\'en\'erale de %convergence faible pour des fonctions croissantes d'ensemble}, {\rm Ann. Scuola Norm. Sup. Pisa %Cl. Sci. (4)}, {\bf 1}, No.4, (1977), 61-99.
\bibitem{Ekeland-Temam}{\sc I. Ekeland \& R. Temam}, {\rm Convex analysis and variational problems}, {\rm Studies in Mathematics and its Applications}, Vol. 1. Amsterdam - Oxford: North-Holland Publishing Company; New York: American Elsevier Publishing Company, Inc. IX, 402 p.,  (1976).
%\bibitem{Evans-Gariepy}{\sc L.C. Evans \& R .F. Gariepy}, {\rm Lecture notes on measure theory %and fine properties of functions}, {\rm CRC Press}, (1992).
%\bibitem{Fonseca-Francfort-Leoni}{\sc I. Fonseca, G. Francfort \& G. Leoni}, {\it Thin elastic %films: the impact of higher order perturbations.}, {\rm Quart. Appl. Math.}, {\bf 65}, No. 1, %(2007), 69-98.
\bibitem{Fonseca-Kinderlehrer-Pedregal_1}{\sc I. Fonseca, D. Kinderlehrer \& P. Pedregal}, {\it Relaxation in $BV \times L^\infty$ of functionals depending on strain and composition}, {\rm
Lions, Jacques-Louis (ed.) et al., Boundary value problems for partial differential equations and applications. Dedicated to Enrico Magenes on the occasion of his 70th birthday. Paris: Masson. Res. Notes Appl. Math.}, {\bf 29}, (1993), 113-152.
\bibitem{Fonseca-Kinderlehrer-Pedregal_2}{\sc I. Fonseca, D. Kinderlehrer \& P. Pedregal}, {\it Energy functionals depending on elastic strain and chemical composition}, {\rm Calc. Var. Partial Differential Equations}, {\bf 2}, (1994), 283-313.
\bibitem{Fonseca-Leoni}{\sc I. Fonseca \& G. Leoni}, {\it Modern Methods in the Calculus of Variations: $L^p$ Spaces}, Springer (2007).
%\bibitem{Fonseca-Muller-quasiconvex}{\sc I. Fonseca \& S. M\"{u}ller}, {\it Quasiconvex %integrands and Lower Semicontinuity in $L^1$}, {\rm Siam. J. Math. Anal.}, {\bf 23}, No. 5, %(1992), 1081-1098.
\bibitem{Fonseca-Muller-relaxation}{\sc I. Fonseca \& S. M\"{u}ller}, {\it Relaxation of quasiconvex functionals in $BV(\Omega;\mathbb R^p)$ for integrands $f(x,u, \nabla u)$}, {\rm Arch. Rational Mech. Anal.}, {\bf 123}, (1993), 1-49.
\bibitem{Fonseca-Rybka}{\sc I. Fonseca \& P. Rybka}, {\it Relaxation of multiple integrals in the space $BV(\Omega, \mathbb R^p)$}, {\rm Proc. R. Soc. Edinb.}, Sect. A {\bf 121}, No. 3-4, (1992), 321-348.
%\bibitem{Le Dret-Raoult} {\sc H. Le Dret \& A. Raoult}, {\it Variational Convergence for nonlinear %shell models with directors and related semicontinuity and relaxation results},  {\rm Arch. %Rational. Mech. Anal.}, {\bf 154}, (2000), 101-134.
%\bibitem{Marcellini}{\sc P. Marcellini}, {\it Approximation of quasiconvex functions and lower %semicontinuity of multiple integrals}, {\rm Manuscripta Math.}, {\bf 51}, (1985), 1-28.
%\bibitem{Meyer}{\sc Y. Meyer}, {\it Oscillating pattern in image processing %and nonlinear evolution equations. The fifteenth Dean Jacqueline B. Lewis %memorial lectures.}
%{\rm University Lecture Series}, {\bf 22}, Providence, RI. American %Mathematical Society (AMS). (2001)
%\bibitem{OSV}{\sc S. Osher, A. Sol\'e \& L. A. Vese}, {\it Image %decomposition and restoration using total variation minimization and the %$H^{-1}$ norm}, {\rm SIAM Multiscale Model. Simul.}, {\bf 1}, No. 3, %(2003), 349-370 .
%\bibitem{Rockafellar-Wets}{\sc R. T. Rockafellar \& R. J-B. Wets},  {\rm Variational Analysis}, %Springer, (1998).
%\bibitem{ROF}{\sc L. I. Rudin, S. Osher \& E. Fatemi}, {\it Nonlinear total %variation based noise removal algorithms}, {\rm Physica D}, {\bf 60}, No. %14, (1992) 259-268.
%\bibitem{VO}{\sc L. A. Vese  \& S. Osher}, {\it Modeling textures with %total variation minimization and oscillating patterns in image processing}, %{\rm J. Sci. Comput.}, {\bf 19}, No. 13, (2003), 553-572.
%\bibitem{VO2}{\sc L. A. Vese \& S. Osher}, {\it Image denoising and %decomposition with total variation minimization and oscillatory functions}, %{\rm Journal of Mathematical Imaging and
%Vision}, {\bf 20}, (2004), 7-18.
%\bibitem{VL}{\sc Vol'Pert A. I., S. I. Hudjaev}, {\it Analysis in Classes of Discontinuous Functions %and Equations of Mathematical Physics}, Martinus Nijhoff Publisher, Dordrecht, (1985).

\end{thebibliography}
\end{document}